\documentclass[11pt,letterpaper,oneside]{amsart}
\usepackage{amsmath}
\usepackage{mathtools}
\usepackage{amsfonts}
\usepackage{amssymb}
\usepackage{amsthm}
\usepackage{amscd}
\usepackage{url}
\usepackage[all]{xy}
\usepackage{caption}
\usepackage{comment} 
\allowdisplaybreaks
\usepackage{verbatim}
\usepackage{xcolor}
\usepackage{tikz-cd}
\usepackage{mathrsfs}
\usepackage{hyperref}
\hypersetup{
    colorlinks=true, 
    linktoc=all,     
    linkcolor=blue,  
}

\usepackage{geometry}

\usepackage{wrapfig}

\DeclareMathOperator{\End}{End}
\DeclareMathOperator{\Lie}{Lie}

\DeclareMathOperator{\Ric}{Ric}

\DeclareMathOperator{\Ham}{Ham}

\DeclareMathOperator{\eq}{eq}

\DeclareMathOperator{\Int}{Int}

\DeclareMathOperator{\Fut}{Fut}
\newcommand{\C}{\mathbb{C}}

\newcommand{\R}{\mathbb{R}}
\newcommand{\ddb}{i\partial \bar\partial}

\newcommand{\mfk}{\mathfrak{k}}
\newcommand{\fham}{\mathfrak{ham}}

\numberwithin{equation}{section}

\newcommand{\cuD}{\mathcal{D}}
\newcommand{\cF}{\mathcal{F}}

\newcommand{\mfh}{\mathfrak{h}}

\newtheorem{theorem}{Theorem}[section]
\newtheorem{lemma}[theorem]{Lemma}
\newtheorem{proposition}[theorem]{Proposition}
\newtheorem{corollary}[theorem]{Corollary}

\theoremstyle{definition}
\newtheorem{example}[theorem]{Example}

\newtheorem{remark}[theorem]{Remark}
\newtheorem{definition}[theorem]{Definition}

\begin{document}
\title{Twisted scalar curvature as a moment map}

\author{R.\ Dervan}
\address[R. Dervan]{Mathematics Institute, University of Warwick, Coventry CV4 7AL, United Kingdom.}
\email{ruadhai.dervan@warwick.ac.uk}
\urladdr{https://www.maths.gla.ac.uk/~rdervan/}
\author{ T.\ Murphy}
\address[T. Murphy]{Department of Mathematics, California State University Fullerton, 800 N. State College Blvd., Fullerton 92831 CA, USA.}
\email{tmurphy@fullerton.edu}
\urladdr{http://www.fullerton.edu/math/faculty/tmurphy/}
\author{J.\ Ross}
\address[J. Ross]{Department of Mathematics, Statistics, and Computer Science, University of Illinois at Chicago, 322 Science and Engineering Offices (M/C 249), 851 S. Morgan Street, Chicago, IL 60607.}
\email{juliusro@uic.edu}
\urladdr{https://sites.google.com/uic.edu/juliusross/home}
\author{L.\ M.\  Sektnan}
\address[L.M. Sektnan]{Department of Mathematical Sciences, University of Gothenburg, 412 96 Gothenburg, Sweden.}
\email{sektnan@chalmers.se}
\urladdr{https://sites.google.com/cirget.ca/lars-sektnan/}
\author{X.\ Wang}
\address[X. Wang]{Department of Mathematics and Computer Science, Rutgers University at Newark, NJ 07102, USA.}
\email{xiaowwan@newark.rutgers.edu}
\urladdr{https://sites.rutgers.edu/xiaowei-wang/}

\begin{abstract} 

We develop the moment map theory of the twisted scalar curvature of a K\"ahler metric. Primarily, we introduce a coupled system of equations on a holomorphic submersion intertwining the twisted scalar curvature of a K\"ahler metric on the base and the fibrewise scalar curvature of a relatively K\"ahler metric on the total space. This resulting system can be viewed as producing the natural coupled metric geometry of holomorphic submersions, and we show that this system appears canonically as a moment map. The approach generalises to  foliations, where we prove similar results.

\end{abstract}

\maketitle
\section{Introduction}

A foundational idea in complex geometry is that geometric PDEs arise in a canonical manner in the study of given geometric structures. The most prominent structure where this takes place is that of complex manifolds themselves, where the natural geometric PDE asks for the scalar curvature of a K\"ahler metric to be constant (producing a \emph{cscK} metric). The search for cscK metrics is then connected with stability conditions in algebraic geometry through the Yau--Tian--Donaldson conjecture, and is further related to the moduli theory of complex manifolds.  In this work, we develop an analogous framework in three related complex-geometric settings:   maps between complex manifolds,  submersions, and  foliations. In each of these three settings the twisted scalar curvature plays a crucial role.  This augments the usual scalar curvature $S(\omega)$ of a K\"ahler metric $\omega$ by subtracting the contraction $\Lambda_{\omega}\alpha$ of an auxiliary closed $(1,1)$-form $\alpha$; the form $\alpha$ is typically  assumed to possess positivity properties related to the geometry of the setup. 

Beginning in the simplest case,  namely that of maps between complex manifolds, we demonstrate that the twisted scalar curvature arises as the natural geometric PDE associated to a K\"ahler metric $\omega$ on the source manifold, where the twist $\alpha$ is induced by pulling back the K\"ahler metric from the target (so in particular $\alpha$ is semipositive). The connection between the setting of maps and the settings of submersions is that given a holomorphic submersion, one obtains a map from the base to the moduli space parametrising the fibres, and this moduli space admits a Weil--Petersson form. In this second setting, however, we demonstrate that the natural metric geometry in fact involves a novel \emph{coupled system} involving both the fibrewise scalar curvature and the twisted scalar curvature, where the twist is given by the Weil--Petersson form. The third setting, namely foliations, is a generalisation of the second, as a holomorphic submersion induces in a natural way a holomorphic foliation which further has compact leaves. Here again a natural coupled system intertwining  leafwise and transverse geometry is derived, governing the metric geometry of foliations.

The way in which the scalar curvature is derived as the natural geometric PDE associated to a K\"ahler metric on a compact complex manifold is through the classical work of  Donaldson \cite{SD-moment} and Fujiki \cite{AK-moment}, where the scalar curvature is viewed as a moment map on the space of almost complex structures on a fixed compact symplectic manifold. The present paper is thus devoted to the construction of moment maps in our three settings, namely maps, submersions and foliations. We take two approaches to this; first in an infinite-dimensional framework modeled on the Donaldson--Fujiki picture, and secondly in a finite-dimensional framework  modeled on the more recent work of the first author  and Hallam \cite{dervan2024universalstructuremomentmaps}.

Most of our work is in the setting  of submersions, and so we next explain that setting in some detail.   Consider a symplectic submersion  $$\pi: (M,\omega_M) \to (B,\omega_B)$$ between compact manifolds of dimensions $2(m+n)$ and $2n$ respectively, so that $\omega_M$ is a closed $2$-form on $M$  which is symplectic on each fibre of $\pi$, and $\omega_B$ is a symplectic form on $B$. Let $\mathscr J^{\Int}_{\pi}$ be the space of pairs of integrable almost complex structures $(J_M,J_B)$ on $M$ and $B$ respectively, compatible with the various structures and making  $\pi$ holomorphic. The space $\mathscr J^{\Int}_{\pi}$  admits a semipositive closed $(1,1)$-form $\Omega$ on $\mathscr J^{\Int}_{\pi}$ defined through an $L^2$-pairing in the spirit of Donaldson--Fujiki.

The space $\mathscr J^{\Int}_{\pi}$ further admits a natural action of a symplectic gauge group $\mathcal G$, given by the diffeomorphisms of $M$ and $B$ preserving the respective forms, in a Hamiltonian fashion  compatible with the map $\pi$.  Since our setup assumes the action of $\mathcal G$ is Hamiltonian, we obtain tautological moment maps $\mu_M: M \to \Lie\mathcal G^*$ and $\mu_B: B \to \Lie\mathcal G^*$. We note here that we do not assume that $\omega_M$ is globally nondegenerate, but merely fibrewise nondegenerate; throughout  this work nondegeneracy is not assumed  in the definition of a moment map. The semipositive form $\Omega$ on $\mathscr J^{\Int}_{\pi}$ is then also $\mathcal G$-invariant, and so it makes sense to ask for a moment map.

\begin{theorem}\label{intromainthm}A moment map $\mu: \mathscr J^{\Int}_{\pi} \to \Lie \mathcal G^*$ for the $\mathcal G$-action on $(\mathscr J_{\pi},\Omega)$ is given by 
\begin{align*}\langle \mu, v \rangle(J_M,J_B) =& \int_M \langle \mu_M, v \rangle (S_{V}(\omega_M,J_M) - \hat S_{V})\omega_M^m\wedge \omega_B^n\\
&+ \int_B \langle \mu_B,v\rangle (S(\omega_B,J_B) - \Lambda_{\omega_B} \alpha_{\pi} - \hat S_{\pi})\omega_B^n.\end{align*}
\end{theorem}

To unwind the notation used above, $J_M$ induces an almost complex structure on each fibre  $(M_b,\omega_b) = \pi^{-1}(b),$ $S_{V}(\omega_M,J_M)$ is the vertical (or fibrewise) scalar curvature of the resulting relatively K\"ahler metric, and $\hat{S}_{V}$ its average.  The family of K\"ahler manifolds
$\pi: (M,\omega_M,J_M) \to (B,J_B)$ induces a Weil--Petersson form on $B$ denoted by $\alpha_{\pi}$, with corresponding average twisted scalar curvature $\hat{S}_{\pi}$.  

 As mentioned above, we also prove a version of this result for a finite-dimensional family of holomorphic submersions (see Theorem \ref{thm:findimmomentmap}). Our proofs in the finite-dimensional and infinite-dimensional settings differ considerably: the former involves more direct tools from equivariant differential geometry, while the latter involves a more indirect adiabatic limit (or asymptotic expansion) argument, which is tightly tied to the geometry of submersions.

In both approaches we obtain a \emph{coupled} system of equations, intertwining in a natural way the fibrewise and basic geometries of the submersion. It is important to observe that this is \emph{not} in a natural way a sum of moment maps, and both terms are crucial in obtaining the moment map property. Zeroes of the moment map naturally correspond to fibrewise cscK metrics, such that the resulting twisted scalar curvature on $B$ is orthogonal to a naturally induced  function space (c.f.\ Remark \ref{interpretationzerosubmersion}).  The form $\Omega$ is merely \emph{semi}-positive; it becomes strictly positive on a natural quotient of $\mathscr J^{\Int}_{\pi}$ where essentially ignore the mixed term in $J_M$ (as we explain in Section \ref{sec:formal}). 

We mention that there is a strong parallel between the theory we develop and the coupled K\"ahler Yang--Mills equations of \cite{coupled}: their geometric setting is a holomorphic vector bundle on a compact K\"ahler manifold, and their moment map setup induces a coupled system of equations intertwining the geometry of the bundle and the geometry of the base. Our system shares many of the features established in their setting, intertwining instead the relative geometry of a submersion with the geometry of the base. We note in addition two other coupled systems related to our work: firstly the coupled cscK equations \cite{HWN,DPingali}, where both equations involved are twisted cscK equations; the  geometric role played by these equations remains not fully understood. Secondly, the constant scalar curvature equation with $B$-field, which involves coupling a variant of the twisted cscK equation with the deformed Hermitian Yang--Mills equation \cite{Schlitzer-Stoppa, SScarpa}; in dimension two, the equations reduce to the genuine twisted cscK equation coupled with the $J$-equation. The geometric motivation for constant scalar curvature equation with $B$-field is to produce canonical representatives of complexified K\"ahler classes, as play a substantial role in mirror symmetry.

Our primary motivation for the perspective we take comes from the geometric role the twisted scalar curvature plays in certain aspects of complex geometry. This arises most notably through work of Fine \cite{finecsck} (where it relates to cscK metrics on the total space of submersions and which is where the twisted scalar curvature first arose), Song--Tian \cite{songTianKRF} (in the analytic approach to the minimal model programme through the K\"ahler--Ricci flow), and Gross--Wilson \cite{Gross_Largecomplexstructures} (where it arises on the base of Calabi--Yau fibrations; see e.g.\ Tosatti \cite{Tosatti_Adiabatic} for further developments). In each of these settings, the twisted scalar curvature appears on the base of submersions where the fibres admit cscK metrics (though singularities are typically present in the latter two).

A moment map interpretation of the twisted scalar curvature was presented, motivated in part by these results, by Stoppa \cite{stoppatwisted}, essentially by summing the moment map interpretations for the scalar curvature and the ``twisting'' term $\Lambda_{\omega}\alpha$. This work motivates connections between twisted scalar curvature and algebro-geometric stability conditions, and as part of this work, Stoppa produced the first algebro-geometric obstruction to the existence of such metrics. From the geometric perspective taken here, the interpretation given by Stoppa appears less related to the role the twisted scalar curvature has played in the theory of submersions, which is our primary motivation. We note in addition that in the work of Stoppa one does not see a coupled system as the natural moment map. In particular, our approach geometrises the appearance of the twisted scalar curvature within the theory of submersions,  and in particular its role in connection to the geometry of the base of families. This appearance has traditionally been viewed as arising from the resulting map to the moduli space of cscK manifolds from the base of a submersion, and our work explains this from a moment map perspective. 

A related motivation for our work is the understanding of the global geometry of the K\"ahler manifold $(M,J_M,\omega_M+k\omega_B)$ for $k\gg 0$, which is the ``adiabatic limit'' \cite{finecsck, Hong_ruled}. Here,  from an analytic perspective, the geometry splits into three pieces: the fibrewise geometry, the base geometry of the map from the base to the moduli space parametrising the fibres, and the global ``twisting'' geometry of the total space. This third geometry corresponds to the theory of optimal symplectic connections \cite{dervansektnan_osc, ortu_optimal}, generalising the Hermite--Einstein condition when $M=\mathbb P(E)$ is the projectivisation of a vector bundle, and a satisfactory moment map interpretation for this equation has been provided by Ortu \cite{ortu_mmap} extending the classical Atiyah--Bott picture. 

What our work suggests is that one cannot decouple the remaining two geometric aspects, and that to obtain a satisfactory moment map interpretation one must actually consider a coupled system intertwining the fibrewise and basic geometry. This necessity is closely related to algebro-geometric counterparts of these various conditions, where the basic analogue of the Yau--Tian--Donaldson conjecture involves a notion of algebro-geometric stability for the map from the base to the moduli space \cite{dervan_ross_stablemaps}.  It is also closely related to an observation of the first and third authors \cite[Remark 4.7]{dervan_ross_stablemaps} in which it is more natural to consider this moduli space as a moduli stack, which can loosely be viewed as a corresponding algebro-geometric coupling.

Moment map interpretations have many applications, as they suggest deep links with algebraic geometry and further to moduli theory. In a more practical direction, moment map results are closely related to (and imply) the existence of Lie algebra characters on a fixed structure (so for us, on a fixed holomorphic submersion), in the same manner as the classical Futaki invariant associated to holomorphic vector fields on a compact K\"ahler manifold \cite{Futaki_anobstruction}. We construct an analogous coupled invariant, associated to a holomorphic vector field on $X$, which we call the transverse Futaki invariant. This invariant is defined relative to a choice of K\"ahler metric on the base; we prove:

\begin{theorem}
\label{thm:futakiinvIntro}
The transverse Futaki invariant only depends on the cohomology class of the K\"ahler metric on the base.
\end{theorem}

As well as being a formal consequence of our moment map results, we also give a direct argument, analogously to the original work of Futaki. We do so for two reasons: firstly, this result will motivate some of the basic ideas of our moment map results, particularly in our adiabatic arguments; secondly, the concrete calculations demonstrate some of the basic analyic features of our coupled system, such as its linearisation, that we expect to play a foundational role in the surrounding theory.

In the final part of this work, we turn to the problem of extending our results to a more general setting. Rather than a submersion $\pi: (M,\omega_M) \to (B,\omega_B)$,  we consider a foliation $\mathcal F\subset TM$ along with a leafwise symplectic form $\omega_M$ and a transverse symplectic form $\omega_B$. After defining a natural analogues of the space of integrable almost complex structure, we prove an analogous claim to Theorem \ref{intromainthm}, namely that a coupled system involving both the leafwise and transverse scalar curvature appears as the natural moment map; Theorem \ref{thm:futakiinvIntro} is similarly proven in the generality of foliations.

\subsection*{Organisation of the paper} We begin in Section \ref{sec:prelims} with background material concerning equivariant differential geometry, and review how to view the scalar curvature as a moment map. We consider two frameworks, the first through ideas from equivariant differential geometry (where what we explain is primarily finite dimensional), and the second from the classical Donaldson--Fujiki approach (which is infinite dimensional).  In Section \ref{sec:maps}, we prove variants of these moment map results for maps between complex manifolds, both in the finite-dimensional and infinite-dimensional frameworks; these are relatively straightforward applications of ideas from equivariant differential geometry. Section \ref{sec:findimsubm} relates to the geometry of submersions, and proves the finite-dimensional version of Theorem \ref{intromainthm}. The proofs here use the equivariant differential geometry approach, and in Section \ref{tfsection} we turn to the adiabatic limit approach to derive the analogue of the Futaki invariant. This latter approach is used in Section \ref{sec:subinfdim} to prove Theorem \ref{intromainthm} in its general, infinite-dimensional form. Lastly, in Section \ref{sec:foliations}, we extend our results to the setting of foliations.

\subsection*{Acknowledgements: } This work was conducted through support of the American Institute of Mathematics (AIM) under the SQuaREs program.  We are particularly grateful to AIM and their staff for this support and excellent working environment. The authors thank Jacopo Stoppa for helpful suggestions, and the referee for their comments and detailed report. RD thanks Calum Spicer for several helpful discussions around the birational geometry of foliations.  RD was funded by a Royal Society University Research Fellowship (URF\textbackslash R1\textbackslash 201041), and JR is partially supported by the National Science Foundation (DMS 1749447). LMS is funded by a Starting Grant from the Swedish Research Council (grant 2022-04574). XW research was partially supported by Simons Foundation Grant 631318 and MPS-TSM-00006636.

\section{Preliminary Material}\label{sec:prelims}
\subsection{Equivariant differential geometry}\label{sec:equivariantdiff}
We recall the basic ideas of equivariant differential geometry, and refer to \cite[Chapter 7]{BGV} for further details. Throughout, we let $K$ be a compact Lie group acting on a manifold $M$, and denote by $\mathfrak k$ the Lie algebra of $K$. 

\begin{definition} An \emph{equivariant differential form} is a $K$-equivariant polynomial map $$\beta: \mfk \to \Omega^*(M),$$ where $K$ acts on $\mathfrak k$ by the adjoint action and by pullback on differential forms.\end{definition}

As with differential forms, there is a natural exterior derivative on equivariant differential forms. In the following, as is standard, we identify an element $v\in\mfk$ with its induced vector field on $M$.

\begin{definition}
    We define the \emph{equivariant differential} $d_{\eq}\beta$ of an equivariant differential form $\beta$ by $$(d_{\eq}\beta)(v) = \iota_v(\beta(v)) + d(\beta(v)),$$ and say that $\beta$ is \emph{equivariantly closed} if $d_{\eq}\beta$ vanishes.
\end{definition}

Suppose $\omega$ is a closed two-form on $M$ and that $K$ acts on $M$ preserving $\omega$. 

\begin{definition} A $K$-equivariant smooth map $\mu: M\to \mathfrak k^* $ is a called a \emph{moment map} for the $K$-action on $(M,\omega)$ if the equivariant differential form $$v\to \omega + \mu^*(v)$$ is equivariantly closed, where $\mu^*(v) =\langle \mu,v\rangle$.
\end{definition}

The reader will easily check that this agrees with the usual definition of a moment map.  Usually one considers moment maps in the case that $\omega$ is a symplectic form, but throughout we allow the definition of a moment map to include $\omega$ that are not necessarily positive.

We will  use the following two properties of equivariant differential forms extensively, the proofs of which are straightforward calculations.

\begin{proposition}\label{prop:EDG} The equivariant differential is compatible with wedge products and fibre integrals:

\begin{enumerate}
\item[(i)] If $\beta, \gamma$ are equivariantly closed,  their wedge product  $$(\beta\wedge\gamma)(v) := \beta(v)\wedge\gamma(v)$$ is equivariantly closed.
\item[(i)] Suppose $\pi: M \to B$ is a $K$-equivariant proper submersion. If $\beta$ is equivariantly closed on $M$, its fibre integral defined by $$\left(\int_{M/B}\beta\right)(v) := \int_{M/B} (\beta(v))$$ is also equivariantly closed.
\end{enumerate}

\end{proposition}

We will use a small number of examples of equivariant differential forms in the present work.

\begin{example}\label{ex:EDG}  Specialising to K\"ahler geometry, the following are examples of equivariantly closed differential forms \cite[Section 4.1]{dervan2024universalstructuremomentmaps}. 

    \begin{enumerate}
    \item[(i)] Suppose $K$ acts holomorphically on a complex manifold $X$, preserving a K\"ahler metric $\omega$, and let $\mu: X \to \mathfrak k^*$ be a moment map for the $K$-action on $(X,\omega)$, so that $\omega+\mu^*$ is equivariantly closed. Then the equivariant differential form $\Ric\omega - \Delta \mu^*$, where $\Delta$ is the Laplacian, is equivariantly closed (where by definition $(\Delta\mu^*)(v) = \Delta (\langle \mu,v\rangle)$). Concretely, along with the natural equivariance condition, this asks that if $\iota_v\omega = -dh$, then $\iota_v\Ric\omega = d\Delta h$. 
    \item[(ii)] We consider a relative version of the previous setting. Thus let $\pi: X\to B$ be a proper holomorphic submersion of relative dimension $m$, which is equivariant with respect to $K$-actions on $X$ and $B$. Endow further $X$ with a $K$-invariant relatively K\"ahler metric $\omega$ (so that $\omega$ is a closed $(1,1)$-form restricting to a K\"ahler metric $\omega_b$ on each fibre $X_c=\pi^{-1}(b)$ of $\pi$). The form $\omega$ induces a Hermitian metric on the relative tangent bundle $V^{1,0} = \ker d\pi$, hence on its top exterior power $-K_{X/B}= \Lambda^mV^{1,0}$, which is the relative anticanonical class; we denote by $\rho \in c_1(-K_{X/B})$ the curvature of this metric. Denote further $\Delta_V$ is the vertical (or relative) Laplacian, defined (on functions) by $$(\Delta_V\phi)(b) = (\Delta_{X_b}\phi|_{X_b})(b).$$ If $\omega + \mu^*$ is equivariantly closed, then the equivariant differential form $\rho - \Delta_V \mu^* $ is equivariantly closed.
    \end{enumerate}
\end{example}

\subsection{The scalar curvature as a moment map}\label{sec2.2}

We recall from \cite[Section 4.2]{dervan2024universalstructuremomentmaps} how to  employ equivariant differential geometry to cast the scalar curvature as a moment map.    In the present section we explain this in the finite-dimensional framework of a proper holomorphic submersion between complex manifolds, and in the subsequent section we consider the infinite-dimensional framework of the space of all almost complex structures compatible with a given symplectic form.

We thus return to the setting of Example \ref{ex:EDG}, so  $\pi\colon (X,\omega)\to B$ is a proper holomorphic submersion of relative dimension $m$, equivariant with respect to the holomorphic action of a compact (finite-dimensional) Lie group $K$, with $\omega$ relatively K\"ahler and $K$-invariant. 

On a given fibre $(X_b,\omega_b)$, with $\omega_b := \omega|_{X_b}$ the restriction (which is assumed K\"ahler), we will be interested in the scalar curvature 
\begin{equation}\label{def:vert-scal}S(\omega_b) = m\frac{\Ric\omega_b\wedge \omega_b^{m-1}}{\omega_b^m},\end{equation} whose (topological) average is given by $$\hat S_b = m\frac{c_1(X_b)\cdot[\omega_b]^{m-1}}{[\omega_b]^{m}},$$ which is independent of $b\in B$. We denote the resulting smooth function by $$S_V(\omega): X \to C^{\infty}(X,\R),$$ which is characterised by $S_V(\omega)|_{X_b} = S(\omega_b).$ This is the vertical (i.e.\ relative or fibrewise) scalar curvature of the relatively K\"ahler metric $\omega$.

\begin{definition} We define the \emph{Weil--Petersson form} on $B$ to be $$\alpha_{\pi} := -\int_{X/B}\rho\wedge\omega^m + \frac{\hat S_b}{m+1}\int_{X/B}\omega^{m+1}.$$
\end{definition}

Thus we obtain a closed, $K$-invariant $(1,1)$-form on the base $B$ of the holomorphic submersion. The Weil--Petersson form $\alpha_{\pi}$ is not necessarily K\"ahler in general, but it is positive in various natural settings. See \cite{fujikischumacher} for foundational results on its positivity, and \cite{fineross_CM} for examples in which it cannot be even semipositive.

Assuming there is a moment map $\mu: X \to \mfk^*$ for the $K$-action on $(X,\omega)$, we have the following, whose proof we recall from \cite{dervan2024universalstructuremomentmaps}.

\begin{theorem}\label{thm:momentmaponB}\cite[Theorem 4.6]{dervan2024universalstructuremomentmaps} A moment map $\nu: B\to\mfk^*$ for the $K$-action on $(B,\alpha_{\pi})$ is given by $$\langle\nu,v\rangle(b) = \int_{X_b}\langle\mu,v\rangle(S(\omega_b)-\hat S_b)\omega_b^m.$$
\end{theorem}

\begin{proof}
By Proposition \ref{prop:EDG} and Example \ref{ex:EDG}, the equivariant differential form $$v\to (\omega+\langle\mu,v\rangle)^m\wedge(\rho-\Delta\langle\mu,v\rangle)$$ is equivariantly closed, so again by Proposition  \ref{prop:EDG} its fibre integral is equivariantly closed on $B$. This fibre integral is a sum of three terms for dimensional reasons; the first is the closed $(1,1)$-form $-\int_{X/B}\rho\wedge\omega^m$; the second (for a choice of $v$) is a function on $B$ whose value at $b\in B$ is $$m\int_{X_b}\langle\mu,v\rangle(\rho_b\wedge\omega_b^{m-1}) = \int_{X_b}\langle\mu,v\rangle S(\omega_b)\omega_b^m$$ where the equality uses the definition of scalar curvature and the fact that $\rho_b = \Ric(\omega_b)$ by construction;  the third is the function on $B$ whose value at $b\in b$ is $$\int_{X_b}\Delta_b\langle\mu,v\rangle\omega_b^m = 0.$$ Thus the equivariant differential form $$v\to -\int_{X/B}\rho\wedge\omega^m + \int_{X/B}\langle\mu,v\rangle S_{V}(\omega)\omega^m,$$ with $S_V(\omega)$ being the vertical scalar curvature defined in Equation \eqref{def:vert-scal}, is equivariantly closed. A similar, easier, calculation holds involving $\frac{\hat S_b}{m+1}\int_{X/B}\omega^{m+1}$ producing the average scalar curvature in the moment map, proving the result by definition of the Weil--Petersson form.\end{proof}

\begin{remark}\label{remark:interpretationzero}
It is worth noting that a zero of the moment map $\sigma$ is not necessarily a point for which $S(\omega_b )-\hat{S}_b$ vanishes, but is merely a point where the function $S(\omega_b )-\hat{S}_b$ is $L^2$-orthogonal to the space $\mu_b^*(\mfk)$ where $\mu^*:\mfk\to X$ denotes the comoment map (that satisfies $\mu_b^*(v) = \langle \mu, v\rangle|_{X_b}$). 
Since we are assuming $K$ is finite dimensional, $\mu_b^*(\mfk)$ is also finite dimensional and so a proper subset of all functions on $X_b$.
\end{remark}

\begin{remark}

This approach to the moment map property of the scalar curvature (which we take in our work on the twisted scalar curvature) has found several applications, such as to the deformation theory of cscK manifolds \cite[Appendix A]{ortu_mmap}, the behaviour of the cscK condition under blowups \cite{DS-blowups} and under the formation of projective bundles \cite{OS-projective}.
\end{remark}

\subsection{The space of almost complex structures}\label{Donaldson-Fujiki-results}
 
We turn to the analogous infinite-dimensional theory. Consider a compact $2m$-dimensional symplectic manifold $(M,\omega)$, and denote by $\mathscr J(M,\omega)$ the space of almost complex structures on $M$ that are compatible with $\omega$.   As explained in \cite[Section 1]{scarpa2021hitchincsck} (to which we refer for details on the present section), the space $\mathscr J(M,\omega)$ can be identified with the space of sections of a certain fibre bundle, and as such it admits a natural structure of a complex Fr\'echet manifold, whose tangent space at a point $J\in \mathscr J(M,\omega)$ is 
$$ T_J \mathscr J(M,\omega) = \{A \in \End(TM) : JA + AJ =0 \text{ and } \omega(AX,JY) = \omega(JX,AY)\}.$$ The natural almost complex structure $\mathbb J$ on $\mathscr J(M,\omega)$ is then given by 
\begin{equation}\mathbb J(A) = JA \text{ for } A\in T_J \mathscr J(M,\omega)\label{eq:holomorphicstructureonmathscrJ};\end{equation} this almost complex structure is integrable, and corresponds to the complex Fr\'echet manifold structure on $\mathscr J(M,\omega)$.

Inside $\mathscr J(M,\omega)$ sits the subspace of integrable almost complex structures, described by the vanishing of the Nijenhuis tensor 
$$ N_J(X,Y) = [X,Y] + J([JX,Y]+[X,JY]-[JX,JY])$$
which we denote by
$$\mathscr J^{\Int}(M,\omega) :=\{J\in \mathscr J(M,\omega) : N_J =0 \}.$$
This is space is possibly singular and since we do not wish, or need, to go into the details of Fr\'echet structures which may not be smooth, we instead work formally declaring that the ``tangent space" to a point $J$ in $\mathscr J^{\Int}(M,\omega)$ consists of those $A\in T_J \mathscr J(M,\omega)$ such that  $N_{J+tA} = O(t^2)$.  Explicitly if we set
$$P_A(X,Y): = J([AX,Y] + [X,AY]) + A[JX,Y]  + A[X,JY] - [AX,JY] - [JX,AY]$$
then
$$T_J\mathscr J^{\Int}(M,\omega) := \{ A\in T_J \mathscr J(M,\omega) : P_A=0\}.$$
A calculation, using $N_J(X,Y)=0$, yields
$$P_{JA}(X,Y) = -JP_A(JX,JY) \text{ for } A\in T_J\mathscr J_{\Int}(M,\omega)$$
which implies that if $A$ is in $T_J\mathscr J^{\Int}(M,\omega)$ then so is $JA$.  That is, $\mathbb J$ formally induces an almost complex structure on $\mathscr J^{\Int}(M,\omega)$.

Now let $\mathcal G$ be the group of exact symplectomorphisms of $(M,\omega)$, so the Lie algebra of $\mathcal G$ is the space $C^{\infty}(M)/\R$ of functions modulo constants which is endowed with the Poisson bracket.  Note that the group $\mathcal G$ group acts on $\mathscr J(M,\omega)$ by conjugation. Define a K\"ahler form $\Omega$ on $\mathscr J^{\Int}(M,\omega)$ whose associated K\"ahler metric is given by 
\begin{equation}\label{mabuchi}
\langle A_1,A_2 \rangle := \int_M \langle A_1, A_2\rangle_{g_J} \omega^m,
\end{equation}
at a point $J \in \mathscr J^{\Int}(M,\omega)$, where the pairing $\langle A_1, A_2\rangle_{g_J}$ is defined using the Riemannian metric $g$ on $M$ induced from $\omega$ and $J$. We note that this K\"ahler structure is actually the restriction of such a structure from the ambient Fr\'echet manifold $\mathscr J(M,\omega)$.

\begin{theorem}\cite{SD-moment,AK-moment}\label{fujiki-donaldson}
A moment map $\mu: \mathscr J^{\Int}(M,\omega) \to \Lie \mathcal G^*$ for the $\mathcal G$-action on $(\mathscr J^{\Int}(M,\omega),\Omega)$ is given by $$\langle \mu, v\rangle(J) = \int_M h_{v}(S(\omega, J) - \hat S)\omega^m,$$ where $\hat S$ is the average scalar curvature, $h_{v}$ is the Hamiltonian of the Hamiltonian vector field $v$, and $S(\omega,J)$ is the scalar curvature of the K\"ahler metric $\omega$ defined using the almost complex structure $J$.    
\end{theorem}

To tie this into the framework of equivariant differential forms, we note that there is a universal family $\mathcal U{\to} \mathscr J(M,\omega)$ whose fibre $\mathcal M_J$ over a point $J\in \mathscr J(M,\omega)$ is the almost complex manifold $(M,\omega,J)$ \cite[Section 5]{dervan2024universalstructuremomentmaps}.  As a Fr\'echet manifold 
$$\mathcal U = M\times \mathscr J(M,\omega)$$
which is given the almost complex structure
$$\mathbb J(v,A) = (Jv, JA) \text{ for } v\in T_x M \text{ and } A\in T_J\mathscr J(M,\omega).$$ In turn this induces an almost complex structure on $M\times \mathscr J^{\Int}(M,\omega)$,
and with this structure it is easy to check that $\pi:\mathcal U\to \mathscr J(M,\omega)$ is holomorphic and further $\mathcal G$-equivariant. 

The universal family $\mathcal U$ admits a natural relatively K\"ahler metric $\omega_{\mathcal U}$, which is simply the pullback $\pi_M^*\omega$ of $\omega$ through the natural projection $\pi_M: \mathcal U\to M$. The action of $\mathcal G$ on $\mathcal U$ is then Hamiltonian, with tautological moment map $\mu_{\mathcal U}: \mathcal U \to \Lie \mathcal G^*$ defined by $\langle \mu_{\mathcal U}, v\rangle = \pi_M^*h_v$. One may then also prove Theorem \ref{fujiki-donaldson} using the geometry of the universal family and the approach of Section \ref{sec2.2} \cite[Section 5]{dervan2024universalstructuremomentmaps}.

\section{Maps}\label{sec:maps}

\subsection{Maps in finite dimensions}\label{sec3.1}

Consider a proper holomorphic submersion of relative dimension $m$ $$\pi: X\to B$$ between complex manifolds $X$ and $B$, not necessarily compact. Let $K$ be a compact Lie group acting holomorphically on $X$ and $B$, making $\pi$ a $K$-equivariant map, and fix a $K$-invariant relatively K\"ahler metric $\omega$ on $X$. Fix an additional K\"ahler manifold $(Y,\alpha)$ and a map $$p: X\to Y$$ that is $K$-invariant, i.e.\ for all $g\in K$, $p\circ g = p.$ We view this data as a $K$-equivariant family of maps $p: (X_b,\omega_b) \to (Y,\alpha)$, parametrised by $B$. We are motivated by the case that $Y$ is a moduli space endowed with a Weil--Petersson form (for example parametrising the fibres of a holomorphic submersion), so we retain the notation $\alpha$ for the K\"ahler metric on the target $Y$.

We will be interested in the \emph{twisted scalar curvature} on each fibre $X_b$, given by 
\begin{equation}
S_{p}(x) := S(\omega_{b}) - \Lambda_{\omega_{b}}p^*\alpha,
\end{equation} 
where $b=\pi(x)$ and we omit that $p^*\alpha$ has been restricted to the fibre $X_b$. We also let $\hat S_{p}$ denote its fibrewise average (which we note is in general different from $\hat S_b$).

\begin{lemma}
The (constant) equivariant differential form $v\to p^*\alpha$ on $X$ is equivariantly closed.
\end{lemma}

\begin{proof} Equivariance of the form follows from $K$-invariance of $p^*\alpha$. So we must show that, for all $v\in\mathfrak k$, $\iota_vp^*\alpha=0.$ Let us write $\zeta_X(v)$ for the real holomorphic vector field on $X$ induced by the $K$-action on $X$. Then $K$-equivariance of $p$ (with $Y$ given the trivial $K$-action) implies $$\iota_{\zeta_X(v)}p^*\alpha= p^*(\iota_{\zeta_Y(v)}\alpha),$$ which vanishes since the $K$-action on $Y$ is trivial, so $\zeta_Y(v)=0$. \end{proof}

We next turn to the base geometry, and view $B$ as being the base of the family of maps $p_b: X_b\to B.$ Define $$\Omega_{\alpha} := \int_{X/B} \omega^m\wedge p^*\alpha,$$ which is a $K$-invariant closed $(1,1)$-form that is semipositive.

\begin{proposition}\label{prop:j-equation}
Suppose $\mu: X\to \mathfrak k^*$ is a moment map for the $K$-action on $(X,\omega)$. A moment map $\nu_{\alpha}: B\to \mathfrak k^*$ for the $K$-action on $(B,\Omega_{\alpha})$ is given by $$\langle \nu_{\alpha}, v\rangle(b) = \int_{X_b} \langle \mu, v\rangle \Lambda_{\omega_b}(p^*\alpha) \omega_b^m.$$

\end{proposition}

\begin{proof} The equivariant differential form on $X$ given by $(\omega+\mu)^m\wedge p^*\alpha$ is equivariantly closed on $X$, so its fibre integral is equivariantly closed on $B$. This implies the result, since $$\int_{X/B}(\omega+\mu)^m\wedge p^*\alpha = \int_{X/B}\omega^m\wedge p^*\alpha + m\int_{X/B}\mu \omega^{m-1}\wedge p^*\alpha$$ and since on each fibre $X_b$ $$m p^*\alpha \wedge\omega_b^{m-1} = \Lambda_{\omega_b}(p^*\alpha)\omega_b^m.$$ \end{proof}

To involve the scalar curvature, set $$\Omega_p := -\int_{X/B}(\rho+p^*\alpha)\wedge\omega^m + \frac{\hat S_{p}}{m+1} \int_{X/B}\omega^{m+1}.$$
\begin{corollary}
A moment map $\nu_p: B \to \mathfrak k^*$ for the $K$-action on $(B,\Omega_p)$ is given by $$\langle \nu_p,v\rangle(b) = \int_{X_b} \langle \mu, v\rangle(S(\omega_b) - \Lambda_{\omega_b}p^*\alpha - \hat S_{p})\omega_b^m.$$
    
\end{corollary}

\begin{proof} Theorem \ref{thm:momentmaponB} and its proof shows that a moment map $\nu: B \to \mathfrak k^*$  with respect to the  
Weil--Petersson form $\Omega_p + \Omega_{\alpha}$
is given by $$\langle \nu,v\rangle(b) = \int_{X_b}\langle \mu, v\rangle (S(\omega_b)  -\hat S_{p})\omega_b^m,$$ so the claim follows from Proposition \ref{prop:j-equation}.\end{proof}

\begin{remark}\label{interpretationzerosubmersion}
As in Remark \ref{remark:interpretationzero}, a zero of the moment map $\nu_p$ is not necessarily a point $b\in B$ where $S(\omega_b) - \Lambda_{\omega_b}p^*\alpha - \hat S_{p}$ vanishes, but rather is a point for which $S(\omega_b) - \Lambda_{\omega_b}p^*\alpha - \hat S_{p}$ is $L^2$-orthogonal to the function space $\mu^*_b(\mathfrak k).$ 
\end{remark}

\subsection{Returning to the space of almost complex structures}
We now discuss an infinite-dimensional analogue of the preceding section.  Consider the universal family $$\mathcal U \to \mathscr J(M,\omega).$$ 
Our discussion will apply verbatim to any  Fr\'echet submanifold of $\mathscr J(M,\omega)$ invariant under an appropriate subgroup of $\mathcal G$ in the same manner, but for simplicity of notation we do not include this generalisation.

Assume in addition that we have a holomorphic map $$p: \mathcal U\to Y$$ with $(Y,\omega_Y)$ a finite-dimensional K\"ahler manifold (where we note that such maps are more abundant over complex submanifolds of $\mathscr J(M,\omega)$). We denote the stabiliser of $p$ under $\mathcal G$ by
$$\mathcal G_p : = \{ g\in\mathcal G : g\circ p = g\},$$ and set
$$\Omega_{\alpha}: = \int_{\mathcal U/\mathscr J(M,\omega)} p^*\alpha \wedge \omega^n$$
which is a semipositive $(1,1)$-form on $\mathscr J(M,\omega)$. Here, we note as in Section \ref{Donaldson-Fujiki-results} we may view the form $\omega$ on $M$ as a relatively (almost) K\"ahler metric on $\mathcal U \to {\mathscr J}(M,\omega)$, where the underlying smooth structure of $\mathcal U$ is $\mathscr J(M,\omega)\times M$ and we simply pull back $\omega$ from $M$). 

The fibre integral defining $\Omega_{\alpha}$ is to be understood in the following sense: since $\mathscr J(M,\omega)$ is a Fr\'echet manifold, given a pair of tangent vectors lying in $T_J\mathscr J(M,\omega)$, we may take a Fr\'echet chart around $J$ with these tangent vectors corresponding to basis vectors of the corresponding Fr\'echet space. We may thus trivially  choose a finite-dimensional submanifold of $\mathscr J(M,\omega)$ containing $J$ and these two tangent vectors, and restricting $\mathcal U$ to this submanifold we may define the pairing of $\Omega_{\alpha}$ with the given tangent vectors in the classical sense. The same reasoning applies to any Fr\'echet submanifold of $\mathscr J(M,\omega)$, finite-dimensional or otherwise.

Setting $$\Omega_p: = \Omega - \Omega_{\alpha},$$ where $\Omega$ is the usual Fujiki--Donaldson K\"ahler metric on $\mathscr J(M,\omega)$ with averaging constant $\hat S_{p}$ rather than the average scalar curvature of the fibres, we obtain the following.

\begin{theorem}\label{thm:morphismtwistedmomentmap}
A moment map of the action of  $\mathcal G_p$ on $\left(\mathscr J(M,\omega), \Omega_p\right)$ is given by
$$ \langle\nu,v\rangle(J) = \int_M h_v \left(S(\omega) -  \Lambda_{\omega} p^*\alpha - \hat S_{p}\right) \omega^m $$
where $S(\omega,J)$ denotes the scalar curvature of the K\"ahler metric defined by $(\omega,J)$ and where $\hat S_{p}$ is the relative average twisted scalar curvature.
\end{theorem}

\begin{proof}
The moment map property for the term $\int_M h_v \Lambda_{\omega} p^*\alpha \omega^m$ follows from the arguments of Section \ref{sec3.1} (which apply equally well to this infinite-dimensional setting), and the result follows from a combination of this and Theorems \ref{thm:momentmaponB} and \ref{fujiki-donaldson}. \end{proof}

\section{Submersions: a finite-dimensional framework}
\label{sec:findimsubm}

Broadly speaking, the previous section constructed moment maps where we incorporated the additional data of a map to some fixed space $Y$. In geometric applications, one is frequently interested in the situation that $Y$  is a moduli space, in which case it is usually more appropriate to consider families parameterised by some space $C$, instead of the the induced map to the moduli space.  In the language of moduli theory, maps from a $C$ to moduli spaces are naturally in  bijection with families over $C$ only when the moduli space is fine (namely when the objects involved have no nontrivial automorphisms).  Since this frequently fails to hold, it is more actually more natural to consider families directly.

We thus turn to families of holomorphic submersions, where we shall see that the geometry is richer than merely involving the twisted scalar curvature.  Our setting shall be a finite-dimensional family of proper holomorphic submersions,  parametrised by $C$ , which we write as the commutative diagram \begin{center}\begin{tikzcd}
X \arrow{r}{\pi}  \arrow{rd}{} 
  & B \arrow{d}{\gamma} \\
    & C.
\end{tikzcd}\end{center} We emphasise here that we assume $\gamma$ and $\pi$ are both proper, but we do not assume $C$ is compact. Thus for each $c\in C $ we obtain a holomorphic submersion, which we write $\pi_c: X_c\to B_c$. We let $\pi$ have relative dimension $m$, and let  $\gamma$ have relative dimension $n$, so that $X_c$ has dimension $m+n$. In a subsequent section, we will return to the analogous infinite-dimensional framework.

Let $K$ be a compact Lie group acting on $X, B$ and $C$, making each map equivariant. Endow $X$ with a $\pi$-relatively K\"ahler metric $\omega_X$ and $B$ with a $\gamma$-relatively K\"ahler metric $\omega_B$, both $K$-invariant.   We also  assume that there exist moment maps $\mu_{B}: B\to\mfk^*$ and  $\mu_{X}: X\to\mfk^*$. The relatively K\"ahler metric $\omega_{X}$ induces a Weil-Petersson form $\alpha_{\pi}$ on $B$, the base of the holomorphic submersion $\pi: X\to B$, defined as before by $$\alpha_{\pi} := -\int_{X/B}\rho_{X}\wedge\omega_{X}^m + \frac{\hat S_b}{m+1}\int_{X/B}\omega_{X}^{m+1},$$ and where $\rho_X\in c_1(-K_{X/B})$ is defined as before through $\omega_X$. A moment map $\nu_{\pi}: B\to\mfk^*$ for the $K$-action on $(B,\alpha_{\pi})$ is given from Theorem \ref{thm:momentmaponB} by $$\langle\nu_{\pi},v\rangle(b) = \int_{X_b}\langle \mu_{X},v\rangle (S(\omega_b)-\hat S_b)\omega_b^m.$$ Here $\omega_b$ is the restriction of $\omega_X$ to the fibre $X_b$ of $\pi: X\to B$, and $\hat S_b$ is the fibrewise average scalar curvature. Said equivalently, the form $\alpha_{\pi}+\nu_{\pi}$ is equivariantly closed on $B$.  

We induce a moment map for the base $C$ of the family of submersions. We obtain two Weil--Petersson forms: $\alpha_{\pi}$ on $B$ and $\alpha_{\gamma}$ on $C$. To define the moment map, denote by $$\hat S_{\pi,c} = \frac{\int_{B_c}(S(\omega_c) - \Lambda_{\omega_c} \alpha_{\pi,c})\omega_c^n}{\int_{B_c}\omega_c^n}$$ the average twisted scalar curvature (using the obvious notation for restriction, so that $\alpha_{\pi,c}$ is the restriction of $\alpha_{\pi}$ to $B_c$ and $\omega_c$ is the restriction of $\omega_C$ to $B_c$) and denote further 
$$\Omega_C = \alpha_{\gamma} + \int_{B/C} \alpha_{\pi}\wedge \omega_B^n - \left(\frac{\int_{B_c}\Lambda_{\omega_c}\alpha_{\pi,c}\omega_c^n}{(n+1)\int_{B_c}\omega_c^n}\right)\int_{B/C}\omega_B^{n+1},$$ which is a $K$-invariant closed $(1,1)$-form on $C$.

\begin{theorem}
\label{thm:findimmomentmap}
A moment map $\sigma_{\gamma}: C\to\mfk^*$ for the $K$-action on $(C,\Omega_C)$ is given by 
\begin{equation*} \langle\sigma_{\gamma},v\rangle(c) = \int_{B_c}\left(\int_{X_c/B_c}\langle\mu_{X},v\rangle(S(\omega_b) - \hat S_b)\omega_{X}^m\right)\omega_{B_c}^n + \int_{B_c} \langle\mu_{B},v\rangle (S(\omega_{c}) - \Lambda_{\omega_c}\alpha_{\pi,c} - \hat S_{\pi,c})\omega_{B_c}^n
\end{equation*}
for $v\in \mfk$ and $c\in C$.
\end{theorem}

\begin{proof}
Note first that Theorem \ref{thm:momentmaponB} implies that the integral involving $S(\omega_c)$ is a moment map with respect to $\alpha_{\gamma}$, and the effect of subtracting $\left(\frac{\int_{B_c}\Lambda_{\omega_c}\alpha_{\pi,c}\omega_c^n}{(n+1)\int_{B_c}\omega_c^n}\right)\int_{B/C}\omega_B^{n+1}$ in $\Omega_C$ is to change the constant in this moment map from $\hat S_c$ to $\hat S_{\pi,c}$. Next, the form $$\int_{B/C}(\alpha_{\pi}+\nu_{\pi})\wedge(\omega_{B}+\mu_{B})^n$$ is equivariantly closed, producing three terms. The first is the $(1,1)$-form $\int_{B/C}\alpha_{\pi}\wedge \omega_{B}^n$, while the second is the function $$c\to \left(n\int_{B/C}\mu_B\alpha_{\pi}\wedge\omega_B^{n-1}\right)(c) = \int_{B_c}\mu_{B}\Lambda_{\omega_{B_c}}\alpha_{\pi,c}\omega_{B_c}^n,$$ and the third  is the function 
$$
c\to \left(\int_{B/C}\nu_{\pi}\omega_{B}^n\right)(c)= \int_{B_c}\left(\int_{X_c/B_c}\langle\mu_{X},v\rangle(S(\omega_{b})-\hat S_{b})\omega_{X}^m\right)\omega_{B_c}^n;$$ equivariant closedness of the resulting equivariant differential form produces the desired moment map result. \end{proof}

We have thus produced a moment map which is a coupled system, involving both the fibrewise average scalar curvature and the twisted scalar curvature on the base.

This in particular proves a finite-dimensional version of Theorem \ref{intromainthm}, establishing a moment map property for this coupled system, involving the twisted scalar curvature of the base and the fibrewise scalar curvature. It is crucial in the argument that the base $C$ is smooth, in order to employ various tools from equivariant differential geometry. In the actual setting of Theorem \ref{intromainthm}, the corresponding base---which is the space of integrable almost complex structures on the submersion---is singular. In fact, it is not even naturally a subspace of a smooth space on which the moment map and group action extend. This means that the tools from equivariant differential geometry do not apply, leading us to take a more robust approach.

\section{The transverse Futaki invariant}\label{tfsection}
Our interest is in the construction of moment maps in various settings. In order to prove such results in the generality we desire we work with an adiabatic limit. For a fixed submersion, this means comparing the geometry of a submersion $\pi: (X,\omega_X)\to (B,\omega_B)$ to the geometry of $(X, k\omega_X+\pi^*\omega_B)$ for $k\gg 0$. 

The goal of the present section is to develop tools to tackle this class of problem, in a simplified setting. Rather than constructing moment maps, we produce a main outcome of the existence of a moment map by constructing a numerical invariant---the Futaki invariant---associated to a holomorphic vector field preserving the submersion structure. The tools and ideas we develop will then be later reapplied to moment map results.

In fact, the results of the present section also apply in more generality than submersions. Endow  a fixed compact K\"ahler manifold $X$ with a  foliation, where a foliation is defined as  a holomorphic subbundle $\mathcal F\subset TX$ of the holomorphic tangent bundle. The key example is a holomorphic submersion $\pi: X \to B$, where the vertical subbundle $V \subset TX$ given by the kernel of the differential $d\pi$ defines a foliation. While a foliation has leaves which are complex submanifolds, and by Frobenius' theorem locally admits the structure of a submersion, in general foliations are not induced by submersions and do not necessarily have compact leaves. By analogy with our interest in the geometry of the base $B$ of a holomorphic submersion, we are interested in the transverse geometry of a foliation: the geometry orthogonal to the leafwise geometry. 

Postponing the precise definitions for the moment, here we are interested in assigning numerical invariants to holomorphic vector fields on $X$. Fixing such a vector field $v$,  suppose we are given a leafwise K\"ahler metric $\omega$ on $X$ (which, by definition, is a closed $(1,1)$-form which is positive on each leaf) and a transverse K\"ahler metric $\beta$, assumed invariant under the part of $v$ lying in a compact torus.  To this data we associate what we call the \emph{transverse Futaki invariant}, which is defined as a sum of $L^2$-pairings of the leafwise average scalar curvature with a fibrewise Hamiltonian, and the twisted transverse  average scalar curvature with a transverse Hamiltonian. In particular, this Futaki invariant vanishes when our system has leafwise constant scalar curvature and constant twisted transverse scalar curvature. Our goal in the present section is to prove, in Theorem \ref{thm:futaki-body}, that the transverse Futaki invariant depends only on the cohomology class of $\beta$. Thus, much as with the classical Futaki invariant, we obtain an equivariant-cohomological, algebro-geometric obstruction to the existence of solutions of our coupled system. We also provide two motivations for our definition of the transverse Futaki invariant, one arising from moment map theory, and the other arising from an adiabatic limit result.

\subsection{Transverse K\"ahler geometry}

We turn to the foundational setup of the K\"ahler geometry of foliations. Let $\cF \subset TX$ be a foliation of rank $m$ and co-rank $n$ (where our notation means that $\cF$ is a holomorphic subbundle of the holomorphic tangent bundle), so that $X$ has dimension $m+n$.  A closed form $\omega \in \Omega^{1,1}(X)$ is defined to be \emph{leafwise K\"ahler} on $(X,\cF)$ if the restriction of $\omega$ to $\cF {\oplus} \cF$ is positive. Given such a leafwise K\"ahler metric $\omega$ on $(X,\cF)$, the orthogonal complement $H \subset TX$ is well-defined, where the  fibre $H_x$ over $x \in X$ is given by
$$H_x = \{ w \in T_x X : \omega(w,v) =0 \textnormal{ for all } v \in \cF_x \}.$$
The splitting  $TX = \cF \oplus H$  naturally induces the  splitting  $\Omega^2(X) = \Omega_{\cF} \oplus \Omega_{H} \oplus \Omega_{\textnormal{mix}}$ of $2$-forms into purely leafwise, purely horizontal, and mixed type. For a two-form $\eta$ we denote 
$$
\eta = \eta_{\cF} + \eta_H + \eta_{\textnormal{mix}}
$$
 the decomposition of $\eta$ into these types. Observe that $\omega = \omega_{\cF} + \omega_{H}$; the mixed term vanishes as $\omega$ itself is used to define the splitting. Moreover,
$\Lambda^{m+n} TX \cong \Lambda^m \cF \otimes \Lambda^n H$. 
\begin{definition}
The \emph{leafwise Ricci curvature} $\rho$ of $\omega$ is the curvature of the metric induced by $\omega$ on $\Lambda^m \cF$. 
\end{definition}

Note that the mixed and horizontal components of $\rho$ may be nontrivial, while its purely vertical component is characterised by its restriction to each leaf, where it agrees with the Ricci curvature of the induced K\"ahler metric on the given leaf.

The definition of a foliation as a subbundle of the tangent bundle implies that there is a canonical notion of a transverse differential form: 

\begin{definition} A differential form $\beta \in \Omega^k (X)$ is  \emph{transverse} if $\iota_{v} \beta = 0$ and $\iota_{v} d\beta = 0$ for all vector fields $v \in \Gamma(X, \cF)$.\end{definition}

The condition that $\beta$ be transverse is equivalent to asking that $\iota_v\beta$ and $\mathcal L_v\beta$ both vanish.  Another general characterisation is that $\beta$ is transverse if and only if $\beta = \beta_H$, where we note it follows from the preceding characterisation that $\beta$ being transverse is independent of choice of the horizontal subspace $H$ defined by $\omega$. Transverse forms are sometimes also called basic forms, notably in the theory of fibre bundles and principal bundles. The set of transverse functions (i.e. functions $f$ with $v(f) = 0$ for all $v \in \Gamma(X, \cF)$, which simply means $f$ is constant along leaves) will be denoted $C^{\infty}_H (X)$. We denote by  $\Omega_H^k(X)\subset \Omega^k(X)$ the vector space of transverse differential forms.

\begin{definition}
A transverse $(1,1)$-form $\beta$ is said to be a \emph{transverse K\"ahler form} if $\beta(\cdot, J(\cdot))$ is positive-definite on $H\times H$ and $\beta$ is closed.
\end{definition}

The associated curvature quantity is as follows, where we view $H$ as a holomorphic vector bundle over $X$ by defining it as the quotient $TX/\cF$.

\begin{definition} The \emph{transverse Ricci curvature} $\Ric\beta$ of a transverse K\"ahler form $\beta$ is  defined to be the curvature of the induced metric  $\Lambda^nH$. The \emph{transverse scalar curvature} $S(\beta)=\Lambda_{\beta}(\Ric\beta)$ is defined to be the horizontal contraction of the transverse Ricci curvature.\end{definition}

Explicitly, the horizontal contraction is characterised  for $\eta\in \Omega^2(X)$ by $$(\Lambda_{\beta}\eta)\beta^n = n\eta_H\wedge\beta^{n-1},$$ so that the transverse scalar curvature is given by $S(\beta) = \Lambda_{\beta} \Ric\beta$. In addition, the $\partial \bar\partial$ lemma admits the following refinement in the foliated setting.

\begin{lemma} If $\beta, \beta'$ are transverse K\"ahler metrics with $[\beta] = [\beta']$, then $\beta - \beta' = i\partial \bar\partial \phi$ with $\phi \in C^{\infty}_H$.
\end{lemma}

 We note that usual $\partial \bar\partial$-lemma implies that such a $\phi$ exists, and as the difference of two transverse differential forms, it follows that $\ddb\phi\in \Omega^2_H(X)$.

\begin{proof} We must show that $d\phi = (d\phi)_H$, i.e. that $(d\phi)_{\cF} = 0$. Since $\ddb\phi$ is transverse, its vertical part is zero and hence
    $\ddb\phi \wedge \omega^{m-1} \wedge \beta^n = 0.$    We then calculate
    \begin{align*}
        \int_X \phi \ddb(\phi) \wedge \omega^{m-1} \wedge \beta^n &= - \int_X i \partial \phi \wedge \bar \partial \phi \wedge \omega^{m-1} \wedge \beta^n, \\ &=- \int_X |(d \phi)_{\cF}|_{\omega}^2 \omega^m \wedge \beta^n.
  \end{align*}
   As $\ddb\phi \wedge \omega^{m-1} \wedge \beta^n$ vanishes, it follows that $(d\phi)_{\cF}=0$, as required.
\end{proof}

Such transverse $\phi$ will be called transverse K\"ahler potentials. Given a transverse K\"ahler form $\beta$, the transverse Laplacian defined on functions by  $\Delta_{\beta} f  = \Lambda_{\beta} ( (\ddb f)_H)$ will play a central role. 
\begin{lemma}  The transverse Laplacian satisfies
$$
\Delta_{\beta} f \omega^m \wedge \beta^n = n \ddb f \wedge \omega^m \wedge \beta^{n-1}.
$$
\end{lemma}
\begin{proof}
By definition of the horizontal contraction, this is a consequence of the equality 
$$\ddb f \wedge \omega^m \wedge \beta^{n-1}= (\ddb(f))_H \wedge \omega^m \wedge \beta^{n-1}.$$\end{proof}

We also give an interpretation of $\Delta_{\beta}$ in terms of adjoints. Denote by $d^* : H^* \to C^{\infty}_H$ be the formal adjoint of $d$ with respect to $\beta$.

\begin{proposition}\label{adjCH}
    If $f \in C^{\infty}_H$, then $\Delta_{\beta} f = - \frac{1}{2} d^* d f$.
\end{proposition}
\begin{proof}
    We prove this as a consequence of the corresponding statement for K\"ahler metrics, which is a strategy we  later employ several times. At each point $x \in X$, the form $\beta_x$ is positive on $H_x$ and trivial on $V_x$, while $\omega_x$ is positive on $V_x$ and makes $H_x$ and $V_x$ orthogonal. It follows that there exists a $k\gg 0$ such that $\omega_x+k\beta_x$ is positive at $x$, and since $X$ is compact, $k$ can be taken independently of $x\in X$ to produce a K\"ahler metric $\omega_k = \omega + k \beta$.
    
    We may thus take  $k$ to be sufficiently large so that $\omega_k = \omega + k \beta$ is K\"ahler, and hence $\omega_{k'}$ is also K\"ahler for $k'>k$. Denoting $\Delta_k$ the Laplacian associated to $\omega_k$,   the usual identities in K\"ahler geometry imply for any $k$
   $$
        \int_X  (\Delta_k f) \phi \omega_k^{m+n}  = \int_X \langle d \phi, d f \rangle_k \omega_k^{m+n}.
   $$
 We  consider the asymptotic expansions in $k$ of the two sides of this equality, and obtain the required equality as the leading order non-zero term of the expansion. 
 
 The Laplacian satisfies
\begin{align*}
        \int_X  (\Delta_k f) \phi \omega_k^{m+n} &=  (m+n) \int_X \ddb f \phi \omega_k^{m+n-1} \\
    &=  (m+n) \binom{m+n-1}{n-1} k^{n-1} \int_X \ddb f \phi \omega^m \wedge \beta^{n-1} + O(k^{n-2})\\
    &=  \frac{m+n}{n} \binom{m+n-1}{n-1} k^{n-1} \int_X (\Delta_{\beta}f) \phi \omega^m \wedge \beta^n + O(k^{n-2}),
\end{align*}
    where the second line follows because $f$ is transverse. The inner product of gradients satisfies
    $$
        \langle d \phi, d f \rangle_k = k^{-1} \langle d \phi, df \rangle_{\beta} + O(k^{-2})$$
    for  $\phi$ and $f$ transverse, implying
    \begin{align*}
        \int_X \langle d \phi, d f \rangle_k \omega_k^{m+n} =&  \binom{m+n}{n} k^{n-1} \int_X  \langle d \phi, df \rangle_{\beta} \omega^m \wedge \beta^n + O(k^{n-2}).
    \end{align*}
    The result follows from the equality between $\binom{m+n}{n}$ and $\frac{m+n}{n} \binom{m+n-1}{n-1}$.
\end{proof}

\subsection{Transverse holomorphy potentials} We next incorporate holomorphic vector fields. Fix  a leafwise K\"ahler metric $\omega$, a holomorphic foliation $\cF \subset TX$, and a transverse K\"ahler metric $\beta$. Now set $k>0$ such that $\omega_k = \omega + k \beta$ is K\"ahler. We consider a holomorphic vector field $v$ whose Lie derivative preserves both $\omega$ and $\beta$ (thus the closure of the flow of the vector field induces a compact torus of biholomorphisms of $X$). A classical result in K\"ahler geometry then states that if $v$ is a holomorphic vector field on $X$ which vanishes somewhere, then it admits a \emph{holomorphy potential} with respect to $\omega_k$, namely a complex-valued function $h_k$ such that $\iota_{v} \omega_k = \bar\partial h_k$, or equivalently $\nabla^{1,0}_k h_k = v$ where $\nabla_k$ is the usual gradient with respect to $\omega_k$ and we take the $(1,0)$-component.

\begin{definition}
\label{def:transverseholpot}
    We say that $h : X \to \mathbb{C}$ is a \emph{transversal holomorphy potential} for $v$ with respect to $\beta$ if $\iota_{v} \beta = \bar \partial h.$ 
\end{definition}

 Transversality of $\beta$ implies that the function  $h$ is  transverse, and hence that $h$ is unique up to the addition of a constant. It also follows that the condition for $h$ to be a transverse holomorphy potential can be written $\iota_{v} \beta = \bar\partial_H h$. Linearity of the usual holomorphy potential construction also implies that $h$ is a transverse holomorphy potential precisely if there exists a $k$ such that $h = h_{k+1} - h_k$, up to the addition of a  constant.

We also note that if $v$ and $v'$ are holomorphic vector fields on $X$ with zeros with the same transversal component, then their transverse holomorphy potentials agree up to addition of a constant, as follows from the fact that $\iota_{v} \beta = \iota_{v_H} \beta$ as $\beta$ is transverse. Next, we explain the dependence of the transverse holomorphy potential on the transverse K\"ahler metric.

\begin{lemma}
\label{lem:changeinholpot}
Let $\phi$ be a transverse K\"ahler potential. If $v$ has transverse holomorphy potential $h$ with respect to $\beta$, then $v$ has holomorphy potential $h+v(\phi)$ with respect to $\beta + \ddb\phi$.
\end{lemma}
\begin{proof}
This follows from the analogous statement for K\"ahler metrics (or from a direct calculation). Indeed, choose $k$ sufficiently large so that both $\omega + k \beta$ and $\omega + k(\beta+\ddb\phi)$ are K\"ahler. If the holomorphy potential with respect to $\omega+k\beta$ is $h_k$, then the holomorphy potential with respect to $\omega+k(\beta+\ddb\phi)$ is $h_k+kv(\phi)$. Similarly, the holomorphy potentials with respect to $\omega + (k+1) \beta$ and $\omega + (k+1)(\beta+\ddb\phi)$ are $h_{k+1}$ and $h_{k+1}+(k+1)v(\phi)$, respectively. Thus, the holomorphy potential of $v$ with respect to $\beta + \ddb\phi$ is given by
\begin{align*}(h_{k+1}+(k+1)v(\phi)) - (h_{k}+kv(\phi)) &= h_{k+1}-h_k+v(\phi), \\ 
&= h + v(\phi).\end{align*}\end{proof}

  \subsection{The transverse Futaki invariant} We next define and motivate the definition of the transverse Futaki invariant. The motivation will be twofold: firstly, we will demonstrate that the transverse Futaki invariant appears as the leading order term in an asymptotic expansion of the usual Futaki invariant. Secondly, in the special case of submersions, we will relate the natural quantity arising from moment map theory to our transverse Futaki invariant.

    Before defining the quantities of interest, we require additional notation.   We let $S_{\cF}(\omega)$ denote the leafwise scalar curvature of $\omega_{\cF}$, which is the leafwise contraction of the leafwise Ricci curvature $\rho$ of $\omega$, and let 
    $$\hat{S}_{\cF} =\frac{m \int_X \rho \wedge \omega^{m-1} \wedge \beta^n}{\int_X \omega^{m} \wedge \beta^n}.$$  
    In principle, this quantity depends on the cohomology class of $\beta$, though we note it is independent of this choice in the setting of submersions as it is then the fibrewise average of the scalar curvature.  We additionally define a constant $\lambda$ by 
\begin{equation}\label{gammaeqn}
\lambda =   \frac{n\int_X ( \Ric \beta + \rho)\wedge \beta^{n-1}\wedge\omega^m}{\int_X \omega^m\wedge\beta^n}.
\end{equation}

Now write the holomorphy potential of $v \in \mfh$ with respect to $\omega+k\beta$ as 
\begin{equation}\label{h_k}
h_k = kh + h_\cF.
\end{equation}

    \begin{definition}\label{def:transversefutaki}
The \emph{transverse Futaki invariant} $\Fut: \mfh \rightarrow \C$ is defined by
\begin{align*}
 \Fut(v)&:= \int_X h_{\cF} \left(S_{\cF}(\omega)-\hat{S}_{\cF} \right) \omega^m \wedge \beta^n + \int_X h (\Lambda_{\beta}( \Ric(\beta) + \rho) - \lambda) \omega^m \wedge \beta^n \\
&\ \ \  \ + \frac{n}{m+1}\int_X h \left(S_{\cF}(\omega)-\hat{S}_{\cF} \right) \omega^{m+1} \wedge \beta^{n-1}.
\end{align*}
\end{definition}

\begin{remark}
    In the context of Sasakian geometry, a transversal Futaki invariant has been defined by Boyer--Galicki--Simanca and Futaki--Ono--Wang. The Reeb vector field gives a (real) one-dimensional foliation and the Sasakian condition implies this is transversally K\"ahler. They introduce a transversal Futaki invariant and show it is independent of the basic cohomology class of the transversal K\"ahler metric (\cite[Proposition 5.1]{Boyer} and \cite[Theorem 4.9]{Futaki_transverse}). The relation between the invariants can most easily be seen when the Sasakian structure is \emph{quasi-regular}, in which case the Sasakian manifold $S$ is an orbifold circle bundle over a K\"ahler orbifold $B$, and the foliation is the one associated to the submersion $S \to B$. In this case, their Futaki invariant reduces to the Futaki invariant on $B$, which essentially corresponds to the term $\int_X h\Lambda_{\beta}\Ric(\beta) \omega^m\wedge \beta^n$ in our expression, provided one normalises such that $h$ has integral zero. Our transverse Futaki invariant is more involved, as our geometric motivations are distinct; see Theorem \ref{thm:trans-fut-leading} for a primary motivating result.  We note that in the quasi-regular Sasakian setting, the cone over $S$ gives an isotrivial holomorphic submersion over $B$ with fibre $\C^*$, which due to noncompactness does not fit into our framework. \end{remark}

    

\subsubsection{The transverse Futaki invariant through an adiabatic expansion}\label{sec:5.3.1} We next motivate our definition through the usual Futaki invariant of the K\"ahler metric $\omega_k=\omega + k \beta$  for $k \gg 0$, which is defined as $$\Fut_k(v) = \int_X h_k(S(\omega_k) - \hat{S}_k)\omega_k^{m+n},$$
where $\hat{S}_k$ denotes the average of the scalar curvature $S(\omega_k)$ of $\omega_k$.  Futaki's classical result is that $\Fut_k(v)$ depends on the class of $[\omega_k]$ (and in particular on $k$ itself), but not on the specific representative $\omega_k$. 

In order to understand the asymptotic behaviour of this quantity as $k \to \infty$, we first consider the scalar curvature.

\begin{lemma}
As $k \to \infty$, the scalar curvature of $\omega_k$ satisfies $$    S(\omega_k) = S_{\cF}(\omega) + k^{-1}\left(S(\beta) + \Lambda_{\beta} \rho_H + \Delta_{\cF} (\Lambda_{\beta}\omega)\right) + O(k^{-2}).$$
\end{lemma}

\begin{proof}
This is a pointwise calculation which,  since foliations are locally submersions, follows from the analogous calculation for submersions originally due to Fine \cite{finecsck}.
\end{proof}

Similarly, by the definition given in Equation \eqref{gammaeqn} of $\lambda,$ the average scalar curvature of $\omega_k$  expands as 
$$
\hat{S}_k = \hat S_{\cF} + k^{-1} \lambda + O(k^{-2}).
$$
Using that $h_k = kh + h_\cF$  and that
$$\omega_k^{m+n} = \binom{m+n}{n} k^n \omega^m \wedge \beta^n + k^{n-1} \binom{m+n}{n-1} \omega^{m+1} \wedge \beta^{n-1} + O(k^{n-2}),
$$
we see that
\begin{align}\label{Fut-expan}
    \frac{\Fut_k(v)}{\binom{m+n}{n} k^n} 
&= \binom{m+n}{n}^{-1} k^{-n} \int_X h_k (S(\omega_k) - \hat S_k) \omega_k^{m+n}, \\
&= k\int_Xh (S_{\cF}(\omega) - \hat{S}_{\cF}) \omega^m \wedge \beta^n+ \int_X h_{\cF} (S_{\cF}(\omega) - \hat{S}_{\cF}) \omega^m \wedge \beta^n\nonumber\\
&\ \ + \int_X h \left(S(\beta) + \Lambda_{\beta} (\rho_H) - \lambda + \Delta_{V} (\Lambda_{\beta}(\omega)) \right) \omega^m \wedge \beta^n \nonumber\\
&\ \ + \frac{n}{m+1}\int_X h \left(S_{\cF}(\omega) - \hat{S}_{\cF}\right) \omega^{m+1} \wedge \beta^{n-1} + O(k^{-1}),  \nonumber\\
&= k\int_Xh (S_{\cF}(\omega) - \hat{S}_{\cF}) \omega^m \wedge \beta^n+  \int_X h_{\cF} (S_{\cF}(\omega) - \hat S_{\cF}) \omega^m \wedge \beta^n\nonumber \\
&\ \ + \int_X h \left(S(\beta) + \Lambda_{\beta} (\rho) - \lambda \right) \omega^m \wedge \beta^n \nonumber \\
&\ \ \ + \frac{n}{m+1}\int_X h \left(S_{\cF}(\omega) - \hat S_{\cF}\right) \omega^{m+1} \wedge \beta^{n-1} + O(k^{-1})\nonumber,
\end{align} 
where we have used that, since $h$ is transverse, its vertical Laplacian vanishes, along with a self-adjointness property of the vertical Laplacian.

Thus, the transverse Futaki invariant appears as a subleading order term in an asymptotic expansion. We next show that the leading order term $\int_X h (S_{\cF}(\omega) - \hat S_{\cF}) \omega^m \wedge \beta^n$ in the expansion actually vanishes automatically, meaning that the transverse Futaki invariant governs the Futaki invariant for $k\gg 0$.  This is essentially a consequence of the fact that $h$ is constant along leaves, which would trivially imply the result in the submersion setting, by combining the theorems of Ehresmann and Fubini. In the absence of a global submersion structure, we argue instead using the transverse Laplacian.

\begin{lemma}\label{h-0} The leading order term in the expansion of the Futaki invariant satisfies
    $$\int_X h (S_{\cF}(\omega) - \hat S_{\cF}) \omega^m \wedge \beta^n=0.$$
\end{lemma}
\begin{proof}
We first claim that there is a constant $c$ and a transverse function $\psi$ such that $h=c + \Delta_{\beta}\psi$. By the main result in \cite{elKacimi-transverse}, transversal elliptic operators in general admit a Hodge decomposition, which in our setting gives that $C^{\infty}_H = \ker \Delta_{\beta} \oplus \mathrm{im } \Delta_{\beta}$ because $\Delta_{\beta}$ is self-adjoint. By Proposition \ref{adjCH}, $\ker \Delta_{\beta}$ consists of the constant functions only, and so we can write $h=c + \Delta_{\beta}\psi$, as claimed.

   It follows that our quantity of interest may be written as
    \begin{align*}
    \int_X \Delta_{\beta}\psi (S_{\cF}(\omega) - \hat S_{\cF}) \omega^m \wedge \beta^n &= n \int_X (S_{\cF}(\omega) - \hat S_{\cF}) \omega^m \wedge \ddb\psi \wedge \beta^{n-1},\\
 &=n \int_X i \bar \partial S_{\cF}(\omega) \wedge \omega^m \wedge \partial \psi \wedge \beta^{n-1},\\
&= n m \int_X i \bar \partial \rho_{\cF} \wedge \omega_{\cF}^{m-1} \wedge \partial \psi \wedge \beta^{n-1},\\
&=  -inm \int_X \bar \partial \rho_{mix} \wedge \omega_{\cF}^{m-1} \wedge \partial \psi \wedge \beta^{n-1},\\ 
&= -nm\int_X \rho_{mix} \wedge \omega_{\cF}^{m-1} \wedge \ddb\psi \wedge \beta^{n-1},\\
&= 0.
    \end{align*} Here, we use two facts: first, that the condition $d \rho = 0$  implies that  $d \rho_{\cF} = - d(\rho_{mix}+\rho_H)$, and the integral involving $d\rho_H$ vanishes automatically; second, since $\ddb\psi$ is purely horizontal, the total transverse components of the final integrand exceeds the corank of $\cF$. This proves the result.
\end{proof}

We have now shown the following:

\begin{theorem}\label{thm:trans-fut-leading}
The Futaki invariant of $\omega_k$ satisfies \begin{align*}    \frac{\Fut_k(v)}{\binom{m+n}{n} k^n}  &= \int_X h_{\cF} (S_{\cF}(\omega) - \hat S_{\cF}) \omega^m \wedge \beta^n + \int_X h \left(S(\beta) + \Lambda_{\beta} (\rho) - \lambda \right) \omega^m \wedge \beta^n \\
&\ \ \ + \frac{n}{m+1}\int_X h \left(S_{\cF}(\omega) - \hat S_{\cF}\right) \omega^{m+1} \wedge \beta^{n-1} + O(k^{-1}).\end{align*}
\end{theorem}

\subsubsection{The transverse Futaki invariant through moment maps}
We give a second motivation for our definition of the tranverse Futaki invariant, by comparing it with the moment map geometry of submersions,  namely through Theorem \ref{thm:findimmomentmap}. This result implies the existence of an analogue of the Futaki invariant by general moment map theory: one simply pairs the moment map with a holomorphic vector field on $\pi: X\to B$ itself. We call this quantity the \emph{submersion Futaki invariant}.

\begin{proposition}\label{lem:samefut}
Suppose $\cF$ is the foliation associated to a submersion $\pi: X \to B$. Then the submersion Futaki invariant agrees with the transverse Futaki invariant.
\end{proposition}
\begin{proof}
From Theorem  \ref{thm:findimmomentmap}, the submersion Futaki invariant is given by 
    \begin{equation}\label{submFutinv}
    \int_{B}\left(\int_{X/B}\langle\mu_{X},v\rangle(S(\omega_b) - \hat S_b)\omega_{X}^m\right)\omega_{B}^n + \int_{B} \langle\mu_{B},v\rangle (S(\omega_B) - \Lambda_{\omega_B}\alpha_{\pi} - \hat S_{\pi})\omega_{B}^n.
\end{equation}
Note that in the foliation notation used in this section, $\langle\mu_X,v\rangle =h_{\cF}$, $\langle\mu_{B},v\rangle = h$, $\omega_X = \omega$, $\omega_B=\beta$, $\hat S_b = \hat S_{\cF}$ and that 
\begin{align*}
    \hat S_{\pi} &= \frac{\int_{B}(S(\beta) - \Lambda_{\beta} \alpha_{\pi})\beta^n}{\int_{B}\beta^n}.
\end{align*}

 Hence in our setting
\begin{eqnarray}\label{1st} 
\int_{B}\left(\int_{X/B}\langle\mu_{X},v\rangle(S(\omega_b) - \hat S_b)\omega_{X}^m\right)\omega_{B}^n
= \int_X h_{\cF} (S_{\cF}(\omega) - \hat S_{\cF}) \omega^m \wedge \beta^n
\end{eqnarray}
and
\begin{eqnarray}\label{2nd}
\int_{B} \langle\mu_{B},v\rangle (S(\omega_B) - \hat S_{\pi})\omega_{B}^n = \int_X h \left(S(\beta) - \hat S_{\pi} \right) \omega^m \wedge \beta^n.
\end{eqnarray}
We will later show that in fact $\hat S_{\pi} = \lambda$, but this relies upon the analysis of the final terms.

We next consider 
$ - \int_{B} \langle\mu_{B},v\rangle \Lambda_{\omega_B}\alpha_{\pi} \omega_{B}^n,
$
where we recall 
\begin{align*}
\alpha_{\pi} &= -\int_{X/B}\rho_{X}\wedge\omega_{X}^m + \frac{\hat S_b}{m+1}\int_{X/B}\omega_{X}^{m+1} \\
&= -\int_{X/B} \rho \wedge \omega^m + \frac{\hat S_{\cF}}{m+1} \int_{X/B} \omega^{m+1}.
\end{align*}
The latter of these terms gives rise to the quantity
\begin{align*}
- \int_{B} \langle\mu_{B},v\rangle \Lambda_{\omega_B} \left(\frac{\hat S_b}{m+1}\int_{X/B}\omega_{X}^{m+1} \right) \omega_{B}^n =& - n \int_{B} h \left(\frac{\hat S_{\cF}}{m+1}\int_{X/B}\omega^{m+1} \right) \beta^{n-1}, \\
=& -\frac{n}{m+1} \int_{X} h \hat S_{\cF} \omega^{m+1} \wedge \beta^{n-1}
\end{align*}
as appears in the definition of the transverse Futaki invariant. Finally, we claim that the term 
$$-\int_{B}\langle\mu_{B},v\rangle \Lambda_{\omega_B} \left(\int_{X/B}\rho_{X}\wedge\omega_{X}^m \right) \omega_{B}^n$$ may be written as 
\begin{align}\label{3rd}
\int_X h \Lambda_{\beta} (\rho) \omega^m \wedge \beta^n + \frac{n}{m+1}\int_X h S_{\cF}(\omega) \omega^{m+1} \wedge \beta^{n-1}.
\end{align}
This follows by using the splitting $\rho_X = \rho_{\cF} + \rho_H + \rho_{mix}$ into its purely vertical, purely horizontal and mixed components, which gives that
\begin{align*}
 n \rho_{X}\wedge\omega_{X}^m \wedge \omega_{B}^{n-1} &= n  \rho_{H}\wedge \omega_{\cF}^m \wedge \beta^{n-1} + nm  \rho_{\cF}\wedge \omega_{\cF}^{m-1} \wedge \omega_H \wedge \beta^{n-1}, \\
&= \Lambda_{\beta}(\rho_{H}) \omega_{\cF}^m \wedge \beta^{n} +  n \Lambda_{\omega_{\cF}}(\rho_{\cF})  \omega_{\cF}^{m} \wedge \omega_H \wedge \beta^{n-1}, \\
&= \Lambda_{\beta}(\rho_{H}) \omega^m \wedge \beta^{n} +  \frac{n}{m+1} S_{\cF}(\omega)  \omega^{m+1} \wedge \beta^{n-1}. 
\end{align*}
Integrating over $X$ against $h$ gives \eqref{3rd}.

A consequence of this is that
$$
- \int_{B} \langle\mu_{B},v\rangle \Lambda_{\omega_B}\alpha_{\pi} \omega_{B}^n = \int_X h \Lambda_{\beta} (\rho) \omega^m \wedge \beta^n + \frac{n}{m+1}\int_X h (S_{\cF}(\omega)-\hat S_{\cF}) \omega^{m+1} \wedge \beta^{n-1}.
$$
Moreover, we see that
\begin{align*}
    \hat S_{\pi} =&\frac{\int_{X}(S(\beta) - \Lambda_{\beta} \alpha_{\pi}) \omega^m \wedge \beta^n}{\int_{X}\omega^m \wedge \beta^n} 
\end{align*}
so that 
$$
\hat S_{\pi} - \lambda = \frac{n}{m+1}\int_X h (S_{\cF}(\omega)-\hat S_{\cF}) \omega^{m+1} \wedge \beta^{n-1}
$$
Similarly to Lemma \ref{h-0}, since $S_{\cF}(\omega) - \hat S_{\cF}$ is in the image of the vertical Laplacian, we see that $\int_X h (S_{\cF}(\omega)-\hat S_{\cF}) \omega^{m+1} \wedge \beta^{n-1} =0$ 
and so this give that $\hat S_{\pi}=\lambda$. Combining the above with \eqref{1st} and \eqref{2nd} we see that the expression in Equation \eqref{submFutinv} equals the the transverse Futaki invariant.

\end{proof}

\subsection{Cohomological invariance of the tranverse Futaki invariant}
 We next turn to one of the main results of our work,  which is as follows.

 \begin{theorem}\label{thm:futaki-body}
The transverse Futaki invariant only depends on the basic cohomology class of the transverse K\"ahler form $\beta$.
\end{theorem}

We begin with some preliminary results concerning the linearisation of the transverse scalar curvature. Key to this will be the  transverse version of the classical Lichnerowicz operator. Before defining this operator, we begin with a basic definition.

\begin{definition}
    Let $\beta$ be a transverse K\"ahler form. The \emph{transverse gradient} of a function $f$ is defined to be the unique vector field $\nabla_{\beta} f \in \Gamma(X,H)$  satisfying
    $$\beta(\nabla_{\beta} f, \cdot) = (df)_H.$$
\end{definition}

In what follows, for a section $v$ of the smooth tangent bundle of $X$, we denote by $v^{1,0}$ its $(1,0)$-component, which is a section of the holomorphic tangent bundle $TX$ of $X$.

 \begin{definition}
    The \emph{transverse Lichnerowicz operator} $L_{\beta}$  is defined by 
    $  L_{\beta} f = \cuD_{\beta}^* \cuD_{\beta} f$,     where    $ \cuD_{\beta} f = \bar \partial (\nabla^{1,0}_{\beta} f)$  and the adjoint $\cuD_{\beta}^*$ is taken with respect to $\beta$.
\end{definition}

A transverse holomorphy potential satisfies $v_H=\nabla^{1,0}_{\beta} h$, and since $v$ is holomorphic,  
    $\bar \partial (\nabla^{1,0}_{\beta} h) = 0$. Thus, by self-adjointness, the kernel of the transverse Lichnerowicz operator consists precisely of transverse holomorphy potentials.

We aim to understand how the various terms appearing in the definition of the transverse Futaki invariant, namely Definition \ref{def:transversefutaki}, depend on the transverse K\"ahler structure. We thus modify the transverse K\"ahler metric $\beta$ by adding $t\ddb \phi$ and we will be interested in the derivative of various quantities when $t=0$. Being a pointwise result, the derivation of the following is a consequence of the analogous result in the submersion setting, see e.g. \cite{hashimoto-twistedcscK,keller-twisted,dervansektnan-fibrations}. 
\begin{proposition} \label{prop:lin}
The linearisation of the operator
$$
\phi \mapsto \Lambda_{\beta+\ddb\phi}( \Ric(\beta+\ddb\phi) + \rho_H) 
$$
at zero is the operator $C^{\infty}_H (X) \to C^{\infty}_H (X)$ given by
$$
\phi \mapsto - L_{\beta} \phi + \langle \rho_H, \ddb\phi \rangle_{\beta} + \frac{1}{2} \langle \nabla_{\beta} S(\beta), \nabla_{\beta} \phi \rangle_{\beta}.
$$
\end{proposition}
As in the submersion case, this follows from the following:
\begin{align}
\frac{d}{dt}\bigg|_{t=0}\left(\omega_{\cF}^m \wedge (\beta+\ddb \phi)^n\right) &=  (\Delta_{\beta}\phi ) \omega_{\cF}^m \wedge \beta^n, \label{eq:volderivative} \\
\frac{d}{dt}\bigg|_{t=0}\left( S(\beta_t)\right) &= - \Delta_{\beta}^2(\phi) + \langle \Ric(\beta), \ddb(\phi) \rangle_{\beta}.
\end{align}

 Next we compute the derivatives of the global quantities appearing in the transverse Futaki invariant.

\begin{lemma} Set  $h_t = h + tv(\phi)$ and $\beta_t = \beta + \ddb \phi$ for $\phi \in C^{\infty}_H(X)$. The derivative 
$$\frac{d}{dt}\bigg|_{t=0}\left(\int_X h_{\cF} (S_{\cF}(\omega) - \hat S_{\cF}) \omega^m \wedge \beta_t^n  +\frac{n}{m+1}\int_X h_{t} (S_{\cF}(\omega) -\hat S_{\cF}) \omega^{m+1} \wedge \beta_t^{n-1}\right)$$
is given by 
\begin{equation}\label{e2}
n m \int_X i h_{\cF} \bar \partial ( \rho_{\cF} ) \wedge \omega^{m-1} \wedge \partial \phi \wedge \beta^{n-1}. 
\end{equation}
\end{lemma}
Note that $h_t = h + tv(\phi)$ is the potential with respect to $\beta_t$, by Lemma \ref{lem:changeinholpot}.
\begin{proof}
From Equation \eqref{eq:volderivative} we have
\begin{equation}\label{e1}
\frac{d}{dt}\bigg|_{t=0}\bigg(h_{\cF} (S_{\cF}(\omega) - \hat S_{\cF}) \omega^m \wedge \beta_{t}^n\bigg) = h_{\cF} (S_{\cF}(\omega) - \hat S_{\cF}) \Delta_{\beta}(\phi) \omega^m \wedge \beta^n, 
\end{equation}
and similarly
\begin{align*}
\frac{d}{dt}\bigg|_{t=0}\bigg(h_{t} (S_{\cF}(\omega) -\hat S_{\cF}) \omega^{m+1} \wedge \beta_{t}^{n-1}\bigg) &=  v(\phi) (S_{\cF}(\omega) -\hat S_{\cF}) \omega^{m+1} \wedge \beta^{n-1} \\
    &+ (n-1) h (S_{\cF}(\omega) -\hat S_{\cF}) \omega^{m+1} \wedge \ddb(\phi) \wedge \beta^{n-2}.
\end{align*}

 Equation \eqref{e1} produces
    \begin{align*}
     \int_X h_{\cF} (S_{\cF}(\omega) - \hat S_{\cF}) (\Delta_{\beta}\phi) \omega^m \wedge \beta^n  &= n \int_X h_{\cF} (S_{\cF}(\omega) - \hat S_{\cF}) \omega^m \wedge \ddb \phi \wedge \beta^{n-1}, \\
&= n\int_X i \bar \partial ( h_{\cF} (S_{\cF}(\omega) - \hat S_{\cF})) \wedge \omega^m \wedge \partial \phi \wedge \beta^{n-1}, \\
&= n\int_X i  (S_{\cF}(\omega) - \hat S_{\cF}) \bar \partial ( h_{\cF}) \wedge \omega^m \wedge \partial \phi \wedge \beta^{n-1} \\
&\ \ + n\int_X i h_{\cF} \bar \partial (  S_{\cF}(\omega) ) \wedge \omega^m \wedge \partial \phi \wedge \beta^{n-1}.
    \end{align*}
Using that  $\bar \partial (  S_{\cF}(\omega) ) \wedge \omega^m = \bar \partial (  S_{\cF}(\omega) \omega^m )$
and  $\phi$ is horizontal shows, up to a factor of $n$,  that the second of these terms is 
\begin{align*}
    \int_X i h_{\cF} \bar \partial (  S_{\cF}(\omega) ) \wedge \omega^m \wedge \partial \phi \wedge \beta^{n-1} &= \int_X i h_{\cF} \bar \partial (  S_{\cF}(\omega)  \omega_{\cF}^{m}) \wedge \partial \phi \wedge \beta^{n-1} \\
    &= m \int_X i h_{\cF} \bar \partial (  \rho_{\cF} \wedge  \omega_{\cF}^{m-1}) \wedge \partial \phi \wedge \beta^{n-1} \\
    &= m \int_X i h_{\cF} \bar \partial ( \rho_{\cF} ) \wedge \omega^{m-1} \wedge \partial \phi \wedge \beta^{n-1}.
\end{align*}
For the first term, we will use that for degree reasons 
    \begin{align*}
        0 & = \iota_v \left( \omega^{m+1} \wedge \partial \phi \wedge \beta^{n-1}\right) \\
        & = (m+1)\bar \partial h_{\cF} \wedge \omega^{m} \wedge \partial \phi \wedge \beta^{n-1} + \omega^{m+1} \wedge (\iota_v \partial \phi) \wedge \beta^{n-1} \\
        & \ \ - (n-1) \omega^{m+1} \wedge  \partial \phi \wedge \bar \partial h \wedge \beta^{n-2}
    \end{align*}
as $\iota_v\omega=\bar \partial h_{\cF}$ and $\iota_v \beta = \bar \partial h$. This gives that 
    \begin{align*}
        \int_X i  (S_{\cF}(\omega) - \hat S_{\cF}) \bar \partial ( h_{\cF}) \wedge \omega^m \wedge \partial \phi \wedge \beta^{n-1} & =  - \int_X \frac{i}{m+1}  (S_{\cF}(\omega) - \hat S_{\cF}) (\iota_{v}\partial \phi) \omega^{m+1} \wedge \beta^{n-1} \\
        &\ \ +\int_X \frac{i(n-1)}{m+1}  (S_{\cF}(\omega) - \hat S_{\cF}) \omega^{m+1} \wedge \partial \phi \wedge \bar \partial h \wedge \beta^{n-2}.
    \end{align*}
    After multiplying with $n$, the first of these cancels with the term  
    $$ \frac{n}{m+1}\int_X v(\phi) (S_{\cF}(\omega) -\hat S_{\cF}) \omega^{m+1} \wedge \beta^{n-1}$$
    in $\frac{d}{dt}\bigg|_{t=0}\bigg(h_{\phi} (S_{\cF}(\omega) -\hat S_{\cF}) \omega^{m+1} \wedge \beta_{t}^{n-1}\bigg)$. 
    
    We are thus left with the terms
\begin{align*}
&n m \int_X i h_{\cF} \bar \partial ( \rho_{\cF} ) \wedge \omega^{m-1} \wedge \partial \phi \wedge \beta^{n-1}
\end{align*}
and
\[
\int_X \frac{in(n-1)}{m+1}  (S_{\cF}(\omega) - \hat S_{\cF}) \omega^{m+1} \wedge \partial \phi \wedge \bar \partial h \wedge \beta^{n-2} 
+ \frac{n(n-1)}{m+1} \int_X h (S_{\cF}(\omega) -\hat S_{\cF}) \omega^{m+1} \wedge \ddb\phi \wedge \beta^{n-2}.
\]
    We would like to show that the latter vanishes. Integrating 
    \begin{align*}
        \frac{n(n-1)}{m+1} \int_X h (S_{\cF}(\omega) -\hat S_{\cF}) \omega^{m+1} \wedge \ddb\phi \wedge \beta^{n-2}
    \end{align*}
    by parts and cancelling the term involving the derivative of $h$ with the first term in the equation above leaves the term 
        \begin{align*}
        \frac{n(n-1)}{m+1} \int_X h \bar \partial (S_{\cF}(\omega))\wedge \omega^{m+1} \wedge \partial \phi \wedge \beta^{n-2},
    \end{align*}
    as $\hat S_{\cF}$ is a global constant. Since $\omega$ is closed, $\bar \partial (S_{\cF}(\omega))\wedge \omega^{m+1} = \bar \partial (S_{\pi}(\omega)) \omega^{m+1}$. Decomposing into vertical and horizontal types we therefore see that this term is equal to 
    \begin{align*}
        m n(n-1) \int_X h \bar \partial (\rho_{\cF})\wedge \omega_{\cF}^{m-1} \wedge \omega_H \wedge \partial \phi \wedge \beta^{n-2}.
    \end{align*}
    Since $\rho$ is closed, a short computation similar to Lemma \ref{h-0} implies that this term vanishes again by considering the vertical and horizontal types. 
\end{proof}

We next relate the linearised operator to transverse holomorphy potentials.

\begin{proposition}
\label{prop:imP}
Let $P$ be the operator given by
$$
\phi \mapsto L_{\beta}(\phi)-\langle \rho_H, \ddb(\phi) \rangle_{\beta} - \frac{1}{2} \langle \nabla_{\beta} (\Lambda_{\beta}\rho_H),  \nabla_{\beta}(\phi) \rangle_{\beta}.
$$
If $h$ is a transverse holomorphy potential for a vector field $v$ then 
$$
\int_X P(\phi) h \omega^m \wedge \beta^n + n\int_X  \rho_{mix} \wedge  \iota_{v} \omega^{m} \wedge i \partial \phi \wedge \beta^{n-1} = 0.
$$
\end{proposition}
\begin{proof}
Suppose $h$ is a transverse holomorphy potential for $v$. Then
$\int_X h L_{\beta}(\phi) \omega^m \wedge \beta^n = 0$, 
since $h$ lies in the $L^2$-orthogonal complement to the image of $L_{\beta}$. Moreover
\begin{align*}
\int_X h \langle \rho_H, \ddb\phi \rangle_{\beta}  \omega^m\wedge \beta^n 
&= \int_X h \Lambda_{\beta} (\rho_H) (\Delta_{\beta}\phi) \omega^m\wedge \beta^n\\
&\ \ - n(n-1) \int_X h \rho_H \wedge \ddb\phi \wedge \omega^m\wedge \beta^{n-2}, \\
&= -\frac{1}{2}\int_X \langle \nabla_{\beta}(h \Lambda_{\beta} (\rho_H)), \nabla_{\beta}\phi \rangle \omega^m\wedge \beta^n\\
&\ \ - n(n-1) \int_X h \rho_H \wedge \ddb\phi \wedge \omega^m\wedge \beta^{n-2}.
\end{align*}
Thus 
\[
\int_X h P(\phi)  \omega^m\wedge \beta^n = \frac{1}{2}\int_X \Lambda_{\beta} \rho_H \langle \nabla_{\beta}(h), \nabla_{\beta}\phi \rangle \omega^m\wedge \beta^n + n(n-1) \int_X h \rho_H \wedge \ddb\phi \wedge \omega^m\wedge \beta^{n-2}. 
\]

Since $\beta$ and $\omega$ are closed, applying Stokes' Theorem yields
\begin{align*}
    \int_X h \rho_H \wedge \ddb\phi \wedge \omega^m\wedge \beta^{n-2} &= i\int_X \bar \partial h \wedge \rho_H \wedge \partial \phi  \wedge \omega^m \wedge \beta^{n-2} \\
    &\ \ + i\int_X  h  \bar \partial \rho_H \wedge \partial \phi  \wedge \omega^m \wedge \beta^{n-2} \\
&=  i\int_X \iota_{v} \beta \wedge \rho_H \wedge \partial \phi  \wedge \omega^m\wedge \beta^{n-2} \\
&\ \ + i \int_X  h \bar \partial \rho_H \wedge \partial \phi  \wedge \omega^m \wedge \beta^{n-2} .
\end{align*}

For the latter of these two terms, note that while $\rho_H$ may not be closed, $(d \rho_H)_H = 0$, since this is the purely horizontal part of $d\rho$, which vanishes as $\rho$ itself is closed. Moreover, by examining the horizontal and vertical components, we see that 
$$
\int_X  h \bar \partial \rho_H \wedge \partial \phi  \wedge \omega^m \wedge \beta^{n-2} = \int_X  h (d \rho_H)_H \wedge \partial \phi  \wedge \omega^m \wedge \beta^{n-2} =0.
$$

Next, consider $i \int_X \iota_{v} \beta \wedge \rho_H \wedge \partial \phi  \wedge \omega^m \wedge \beta^{n-1}$. The form $\rho_H \wedge \partial \phi  \wedge \omega^m \wedge \beta^{n-1}$ vanishes for degree reasons, hence
$$
\iota_{v} \left(\rho_H \wedge \partial \phi  \wedge \omega^m \wedge \beta^{n-1}\right) = 0
$$
also. Expanding yields  
\begin{align*}
    (n-1) \iota_{v}\beta \wedge \rho_H \wedge \partial \phi  \wedge \omega^m \wedge \beta^{n-2}  &= - \iota_{v} \rho_H \wedge \partial \phi  \wedge \omega^m \wedge \beta^{n-1} - \rho_H \wedge \iota_{v} \partial \phi  \wedge \omega^m \wedge \beta^{n-1} \\
    &\ \ \ +m \rho_H \wedge \partial \phi  \wedge \iota_v \omega \wedge \omega^{m-1} \wedge \beta^{n-1}, \\
&= -\iota_{v} \rho_H \wedge \partial \phi  \wedge \omega^m \wedge \beta^{n-1} - \rho_H \wedge \iota_{v} \partial \phi  \wedge \omega^m \wedge \beta^{n-1} ,
\end{align*}
where the third term in the first equality again vanishes upon splitting into types.
We now claim that 
$$n \int_X \rho_H \wedge \iota_{v} \partial \phi  \wedge \omega^m \wedge \beta^{n-1} $$
cancels with the inner product term. Indeed,
\begin{align*}
    n \int_X \rho_H \wedge \iota_{v} \partial \phi  \wedge \omega^m \wedge \beta^{n-1} & = \int_X \Lambda_{\beta} (\rho_H) \iota_{v} \partial \phi \omega^m \wedge \beta^{n}, \\
&= \frac{1}{2} \int_X \Lambda_{\beta}(\rho_H) \langle v , \nabla_{\beta} (\phi) \rangle \omega^m \wedge \beta^{n}, \\
&= \frac{1}{2} \int_X \Lambda_{\beta}(\rho_H) \langle \nabla_{\beta} (h), \nabla_{\beta} (\phi) \rangle \omega^m \wedge \beta^{n}.
\end{align*}

The upshot is that 
$$
\int_X h P(\phi)  \omega^m\wedge \beta^n = -n(n-1) \int_X \iota_{v} \rho_H \wedge i\partial \phi  \wedge \omega^m \wedge \beta^{n-1}
$$
and so what remains is to show the vanishing of 
\begin{align*}
    -n(n-1) \int_X \iota_{v} \rho_H \wedge i\partial \phi  \wedge \omega^m \wedge \beta^{n-1} + n\int_X  \rho_{mix} \wedge  \iota_{v} \omega^{m} \wedge i \partial \phi \wedge \beta^{n-1}.
\end{align*}
Using the same argument as in the previous paragraph, we see that 
\begin{align*}
    0&= \iota_{v} \left(\rho_{mix} \wedge  \omega^{m} \wedge i \partial \phi \wedge \beta^{n-1}\right), \\
    &=\iota_{v}\rho_{mix} \wedge  \omega^{m} \wedge i \partial \phi \wedge \beta^{n-1} 
    + \rho_{mix} \wedge  \iota_{v}\omega^{m} \wedge i \partial \phi \wedge \beta^{n-1} \\
    &\ \ + \rho_{mix} \wedge  \omega^{m} \wedge i \iota_{v} \partial \phi \wedge \beta^{n-1} -\rho_{mix} \wedge  \omega^{m} \wedge i \partial \phi \wedge \iota_{v} \beta^{n-1}.
\end{align*}
Note that the terms $\rho_{mix} \wedge  \omega^{m} \wedge i \partial \phi \wedge \iota_{v} \beta^{n-1}$  and $\rho_{mix} \wedge  \omega^{m} \wedge i \iota_{v} \partial \phi \wedge \beta^{n-1}$ vanish since they both have an odd number of horizontal components.

It follows  that the quantity of interest reduces to the negative of
$$
n(n-1) \int_X \iota_{v} \rho_H \wedge i\partial \phi  \wedge \omega^m \wedge \beta^{n-1} + n\int_X  \iota_{v}\rho_{mix} \wedge  \omega^{m} \wedge i \partial \phi \wedge \beta^{n-1} =\int_X (\iota_{v} \rho)_H \wedge i\partial \phi  \wedge \omega^m \wedge \beta^{n-1}.
$$
But $\iota_{v} \rho = d \Delta_{\cF} (h_{\cF})$ and so this equals
\begin{align*}
\int_X \Delta_{\cF}(h_{\cF}) \ddb \phi  \wedge \omega^m \wedge \beta^{n-1}&= m \int_X \ddb(h_{\cF}) \wedge \ddb \phi  \wedge \omega^{m-1} \wedge \beta^{n-1}, \\
&= m\int_X d(i\bar \partial(h_{\cF}) \wedge \ddb \phi  \wedge \omega^{m-1} \wedge \beta^{n-1}), \\
&= 0.
\end{align*}
\end{proof}

We are now ready to show that the transverse Futaki invariant is independent of the representative in the class.
\begin{proof}[Proof of Theorem \ref{thm:futaki-body}]
Define  
\begin{align*}
 \Fut_{\beta}(v)=& \int_X h_{\cF} \left(S_{\cF}(\omega)-\hat S_{\cF} \right) \omega^m \wedge \beta^n + \int_X h (\Lambda_{\beta}( \Ric(\beta) + \rho) - \lambda) \omega^m \wedge \beta^n \\
&+ \frac{n}{m+1}\int_X h \left(S_{\cF}(\omega)-\hat{S}_{\cF} \right) \omega^{m+1} \wedge \beta^{n-1}.
\end{align*}
as the transverse Futaki invariant of a transverse K\"ahler form $\beta$, which we aim to show is independent of $\beta$ in its class.  
 
We claim that for every $\phi \in C^{\infty}_H (X)$ and every $v \in \mfh$, 
$$\frac{d}{dt}\bigg|_{t=0}\left( F_{\beta + t \ddb \phi}(v) \right) = 0,$$ which would give the result.   

Using Equation \eqref{eq:volderivative} and Proposition \ref{prop:lin}, we see that 
\begin{align*}
\frac{d}{dt}\bigg|_{t=0}\left( F_{\beta + t \ddb \phi}(v) \right) = 
  n m \int_X i h_{\cF} \bar \partial ( \rho_{\cF} ) \wedge \omega^{m-1} \wedge \partial \phi \wedge \beta^{n-1} - \int_X P(\phi) h \omega^m \wedge \beta^n .
\end{align*}
We now rewrite the first of these terms.  
Since $\rho$ is closed, we may replace $\rho_{\cF}$ with $-\rho_{mix}$ and obtain that
    \begin{align*}
       n m \int_X i h_{\cF} \bar \partial ( \rho_{\cF} ) \wedge \omega^{m-1} \wedge \partial \phi \wedge \beta^{n-1}   &= -n m   \int_X i h_{\cF} \bar \partial ( \rho_{mix} ) \wedge \omega^{m-1} \wedge \partial \phi \wedge \beta^{n-1}, \\
&= n m \int_X i \bar \partial (h_{\cF}) \wedge  \rho_{mix} \wedge \omega^{m-1} \wedge \partial \phi \wedge \beta^{n-1} \\
& - n m \int_X i h_{\cF}  \rho_{mix} \wedge \omega^{m-1} \wedge \ddb \phi \wedge \beta^{n-1}.
    \end{align*}
Comparing types again we see that the latter term vanishes. Thus what remains is 
\begin{align*}
n m \int_X i \bar \partial (h_{\cF}) \wedge  \rho_{mix} \wedge \omega^{m-1} \wedge \partial \phi \wedge \beta^{n-1} 
    &=n m \int_X i \iota_{v} \omega \wedge  \rho_{mix} \wedge \omega^{m-1} \wedge \partial \phi \wedge \beta^{n-1}, \\
    &=- n \int_X  \rho_{mix} \wedge  \iota_{v} \omega^{m} \wedge i \partial \phi \wedge \beta^{n-1}. 
\end{align*} 

The  relevant derivative is therefore given by
\begin{align*}
  - n \int_X  \rho_{mix} \wedge  \iota_{v} \omega^{m} \wedge i \partial \phi \wedge \beta^{n-1} - \int_X P(\phi) h \omega^m \wedge \beta^n,
\end{align*}
which vanishes by Proposition \ref{prop:imP}, implying the result.
\end{proof}

\section{Submersions: an infinite-dimensional framework}\label{sec:subinfdim}

We return to the moment map geometry of submersions, where we consider the infinite-dimensional geometry to the space of almost complex structures on a submersion. In the subsequent section, we generalise our results to the setting of foliations. We thus define a suitable analogue version of the space of almost complex structures compatible with a smooth symplectic submersion,  introduce a gauge group action on this space, and  prove that a version of our coupled system arises as a moment map. A key tool in our work will be to employ the robust adiabatic arguments of Section \ref{sec:5.3.1} to deduce the moment map property.

\subsection{The space of almost complex structures compatible with a submersion}
We fix a smooth symplectic submersion, by which we mean the data of $\pi:(M,\omega_M)\to (B,\omega_B)$, with $\pi: M \to B$ being a smooth proper submersion between compact manifolds, with $\omega_B$ being symplectic and with $\omega_M$ being relatively symplectic (that is, a closed two-form whose restriction to each fibre is symplectic).

We begin with the relevant space of almost complex structures.

\begin{definition}
The \emph{space of integrable almost complex structures} on $\pi$ is $$
\mathscr J^{\Int}_\pi=\left\{\left. (J_M,J_B)\in\mathscr J^{\Int}(M,\omega_M)\times \mathscr J^{\Int}(B,\omega_B) \right |  J_B\circ d\pi=d\pi\circ J_M\right\}.
$$
\end{definition}

Thus an element of $\mathscr J^{\Int}_\pi$ is a pair of integrable almost complex structures on $M$ and $B$, which make the smooth map $\pi: M \to B$ holomorphic. Our notation $\mathscr J^{\Int}(M,\omega_M)$ is a minor extension of the usual notation, referring to integrable almost complex structures on $M$ which preserve $\omega_M$, which may not be globally symplectic. Recall that the tangent space to $\mathscr J^{\Int}(M,\omega_M)$ at some complex structure $J_M$ consists of the $A_M \in \mathcal{A}^0(\End TM)$ such that $A_M$ is compatible with the complex structure, i.e. $A_M \circ J_M = - J_M \circ A_M$,  is compatible with the symplectic form, i.e. $\omega_M(J_M(\cdot), A_M(\cdot))=-\omega_M(A_M(\cdot), J_M(\cdot))$ and further $\bar \partial_{J_M}(A_M^{1,0})=0$, where $A_M^{1,0}$ is the $(1,0)$-component of $A_M$ with respect to $J_M$. 

We view $\mathscr J^{\Int}_\pi$  as a version of a complex subspace of $\mathscr J(M,\omega_M)\times \mathscr J(B,\omega_B)$, but lacking a robust infinite-dimensional theory of singular complex spaces we give a more concrete treatment and view $\mathscr J^{\Int}_\pi$ as a set with a well-defined (formal) tangent space.

\begin{definition} We define the \emph{tangent space} $T_{(J_M,J_B)}\mathscr J^{\Int}_\pi$ at $(J_M,J_B)$ to consist of   pairs  $(A_M,A_B) \in T_{J_M} \mathscr J^{\Int}(M,\omega_M) \times T_{J_B} \mathscr J^{\Int}(B,\omega_B)$ such that $A_B \circ d \pi = d \pi \circ A_M$. For each $(J_M,J_B) \in \mathscr J^{\Int}_\pi$, we define an almost complex structure on $T_{(J_M,J_B)}\mathscr J^{\Int}_\pi$ by 
$$
(A_M,A_B) \mapsto (J_M \circ A_M, J_B \circ A_B).
$$
\end{definition}
Note that the condition $A_B \circ d \pi = d \pi \circ A_M$ is preserved by $(J_M \circ A_M, J_B \circ A_B)$, so the above almost complex structure sends $T_{(J_M,J_B)}\mathscr J^{\Int}_\pi$ to itself.

The space  $\mathscr J^{\Int}_\pi$ admits a natural gauge group action.

\begin{definition} We define the group of \emph{Hamiltonian symplectomorphisms} $\Ham(\pi)$ of $\pi:(M,\omega_M)\to (B,\omega_B)$ to be pairs $(\phi_M, \phi_B) \in \Ham(M,\omega_M) \times \Ham(B,\omega_B)$ such that $\pi \circ \phi_M = \phi_B$. 
\end{definition}

The Lie algebra $\mathfrak h_{\pi}$ of $\Ham(\pi)$ thus consists of pairs of Hamiltonian vector fields $(v_M, v_B)$, such that the pushforward of $v_M$ under $d\pi$ equals $v_B$.

We next endow $\mathscr J^{\Int}_\pi$ with a natural closed semipositive $(1,1)$-form, by defining for $A^j = (A^j_m,A^j_B) \in T_{(J_M, J_B)}\mathscr J^{\Int}_\pi$,
$$
\Omega_{\mathscr J_\pi}(A^1, A^2):= \int_M\left(\omega_B(A_B^1, A_B^2)+ \omega_M\big|_{V}(A_M^1, A_M^2)\right)\omega_M^{m}\wedge \omega_B^n,
$$
where $\omega_M\big|_{V}$ is $\omega_M$ restricted to the purely vertical component $\End (\ker \pi)$. 

We will show that this is semipositive, but note that it is not necessarily strictly positive. In order to see this, and to construct a moment map, we  embed $\mathscr J^{\Int}_\pi$ into a sequence of infinite-dimensional K\"ahler manifolds.
\begin{lemma}\label{lem:embeddings}
The space $\mathscr J^{\Int}_\pi$ admits a natural sequence of $\mathrm{Ham}(\pi)$-equivariant holomorphic embeddings $$\sigma_k: \mathscr J^{\Int}_\pi \to \mathscr J^{\Int}(M, \omega_M+k\pi^*\omega_B)$$ for $k \gg 0$.
\end{lemma}

\begin{proof}
The map $\sigma_k$ is defined in the natural way: if $(J_M,J_B) \in \mathscr J^{\Int}_\pi$, then $J_M$ is compatible with $ \omega_M+k\pi^*\omega_B$ for all $k$, since $J_B\circ d\pi=d\pi\circ J_M$. To prove that $\sigma_k$ is injective, it suffices to prove that we may recover $J_B$ from $J_M$ and the smooth map $\pi$. But by definition, $J_B\circ d\pi=d\pi\circ J_M$, so since $d\pi$ is surjective ($\pi$ being a submersion), we may recover $J_B$ from $J_M$. A similar argument implies that the formal differential of $\sigma_k$ is injective. To complete the proof, we must show that the $\sigma_k$ are  holomorphic,  which by definition means their differentials intertwine the two almost complex structures; letting $(A_M,A_B) \in T_{(J_M,J_B)}\mathscr J^{\Int}_\pi,$ we see that 
\begin{align*}
    d \sigma_k(J \circ (A_M, A_B)) &= d \sigma_k((J_M \circ A_M, J_B \circ A_B)) \\
    =& J_M \circ A_M \\
    =& J_M \circ d\sigma_k((A_M,A_B))
\end{align*}
and so the $\sigma_k$ are holomorphic.
\end{proof}

The closedness and semi-positivity of $\Omega_{\mathscr J_{\pi}}$ now follows from an adiabatic argument.
\begin{lemma}\label{lem:Omegaclosed}
The form $\Omega_{\mathscr J_\pi}$ is a closed, semipositive $(1,1)$-form on $\mathscr J^{\Int}_\pi$.
\end{lemma}
\begin{proof}
We consider an adiabatic expansion of the usual K\"ahler form $\Omega_k$ on $\mathscr J^{\Int}(M, \omega_M + k \pi^*\omega_B)$, which we have $\mathscr J_{\pi}$ embedded into by Lemma \ref{lem:embeddings}. For $A^j = (A^j_M, A^j_B) \in T_{(J_M,J_B)} \mathscr J^{\Int}_{\pi}$, we have $A^j_M \in T_{J_M} \mathscr J(M, \omega_M + k \pi^*\omega_B)$ and $A^j_B \circ d \pi = d \pi \circ A^j_M$. Moreover, the $A^j_M$ have vanishing $\mathrm{Hom}(V,H)$-component. An expansion of the metric $g_k$ on $\End TM \cong T^*M \otimes TM$ induced by $g_k$ on $TM$ and some linear algebra then gives that 
\begin{align*}
    g_k(A^1_M, A^2_M) &= g_V(A^1_M, A^2_M) + \pi^*g_B (A^1_M, A^2_M) + O(k^{-1}),\\
    &= g_V(A^1_M, A^2_M) + g_B(A^1_B, A^2_B) + O(k^{-1}).
\end{align*}
Here $g_V$ is the restriction of $g_M$ to the purely vertical component $\End V$. There is no $O(k)$-term in this expansion precisely because the $A^j_M$ have vanishing $\mathrm{Hom}(V,H)$-component.

This in turn gives that 
\begin{eqnarray*}
\Omega_k(A^1_M, A^2_M) = \binom{m+n}{n}k^n  \Omega_{\mathscr J_\pi}(A^1, A^2)+O(k^{n-1}).
\end{eqnarray*}
Hence closedness of $\Omega_{\mathscr J_\pi}$ follows from closedness of the $\Omega_k$, while semipositivity is clear (and also follows from the positivity of the $\Omega_k$). 
\end{proof}

We are now in a position to ask for a moment map for the $\Ham(\pi)$-action on $(\mathscr J^{\Int}_\pi,\Omega_{\mathscr J_\pi})$.  
\begin{theorem}
\label{thm:cplxstrmomentmap}
Let $v = (v_M,v_B) \in \mathfrak h_{\pi}$, where $h_M$ is the Hamiltonian for $v_M$ with respect to $\omega_M$ and similarly for $h_B$. A moment map $\mu_{\pi} : \mathscr J^{\Int}_\pi \to \mathfrak h_{\pi}^*$ for the $\mathrm{Ham}(\pi)$-action on $\left(\mathscr J^{\Int}_\pi,\Omega_{\mathscr J_\pi}\right)$ at $v$ is given by
\begin{align*}
\langle\mu_\pi(J),v\rangle &= \int_M h_{M} \left(S_{V}(\omega_M)-\hat S_{V} \right) \omega_M^m \wedge \omega_B^n + \int_M h_B (\Lambda_{\omega_B}( \Ric(\omega_B) + \rho_H) - \hat S_{\pi}) \omega_{M}^m \wedge \omega_B^n \\
 & \ \ \ + \frac{n}{m+1}\int_M h_B (S_{V}(\omega_{M}) -\hat S_{V}) \omega_{M}^{m+1} \wedge \omega_B^{n-1}.
\end{align*}
\end{theorem}

\begin{remark}
This result proves Theorem \ref{intromainthm}, which may be seen by rewriting terms in a similar fashion to the proof of Proposition \ref{lem:samefut}.
\end{remark}

The spaces $(\mathscr J(M, \omega_M+k\pi^*\omega_B),\Omega_k)$ each admit a natural moment map with respect to their respective Hamiltonian symplectomorphism group actions, given by the scalar curvature. This induces a moment map for the subgroup $\mathrm{Ham}(\pi)$ by composing with the map $\mathfrak{h}_k^* \to \mathfrak{h}_{\pi}^*$, dual to the inclusion $\mathfrak{h}_{\pi} \subset \mathfrak{h}_k$, where $\mathfrak{h}$ is the Lie algebra of the exact symplectomorphisms of $(M, \omega_k)$. This simply amounts to restricting the moment map with respect to $\Omega_k$ to the vector fields in $\mathfrak{h}_{\pi}$. Since the moment map on an invariant subspace is simply the restriction of the moment map, we have a $k$-dependent sequence of K\"ahler metrics $\Omega_k$ and a sequence of moment maps on $ \mathscr J^{\Int}_\pi$ itself. It is then straightforward to see that, provided the sequence of K\"ahler metrics and moment maps admit an asymptotic expansion in $k$, the moment map property is preserved for each term in the asymptotic expansion (see for example  \cite[Theorem 3.12]{LSW2022}). This is the main idea in the following proof.

\begin{proof}[Proof of Theorem \ref{thm:cplxstrmomentmap}]
As above, $\Omega_{k}$ is the $k$-dependent K\"ahler form on  $\mathscr J^{\Int}_\pi$ obtained as the pullback of the Donaldson--Fujiki K\"ahler metric on $\mathscr J^{\Int}(M, \omega_k)$ by the injective map $\sigma_k$ considered in Lemma \ref{lem:embeddings}, where $\omega_k = \omega_M+k\pi^*\omega_B$. Recall from  Theorem \ref{fujiki-donaldson} that the moment map $\mu_k : \mathscr J^{\Int}(M, \omega_k) \to \mathfrak{ham}^\ast (M,\omega_k)$ for the $\mathrm{Ham}(M,\omega_k)$-action at  $J \in \mathscr J^{\Int}(M, \omega_k)$ is given by
\begin{align*}
v \mapsto  \int_M h_k \big(S(\omega_k,J)-\hat S(\omega_k,J)\big)\omega_k^{n+m},
\end{align*}
where $h_k$ is the Hamiltonian for $v$ with respect to $\omega_k$. 

For an element $v \in \mathfrak{h}_{\pi}$, we can write $h_k = h_M + k h_B$ as in Equation \eqref{h_k}. It follows from this and the expansion in Equation \eqref{Fut-expan} that at $J\in \mathscr J_{\pi}$, $\mu_k(J)$ admits an expansion in powers of $k$. Now, the moment map property states that, for all  for all $v \in \mfk_{\pi}$
$$
d(\langle \mu_k(J), v \rangle) = \iota_{v} \Omega_k.
$$
From the expansion of $\Omega_k$ given in the proof of Lemma \ref{lem:Omegaclosed}, we  therefore obtain that the $O(k^n)$-term in the expansion of $\mu_k(J)$ satisfies the moment map property with respect to $\Omega_{\mathscr J_\pi}$, up to a factor of $\binom{m+n}{n}$. But from the calculation in Equation \eqref{Fut-expan} the expansion of $\mu_k$ is given by
\begin{eqnarray*}
\frac{\left\langle \mu_{\mathscr J,k}(J), v\right\rangle}{\binom{m+n}{n}\cdot k^n}  &=& \left(\binom{m+n}{n}\cdot k^n\right)^{-1}\int_Mh_k \big(S(\omega_k,J_M)-\hat S(\omega_k,J_M)\big)\omega_k^{n+m},\\
&=&\int_M h_{M} (S_{V}(\omega_M) - \hat{S}_{V}) \omega_{M}^m \wedge \omega_B^n  \\
&& + \int_M h_B \Big(S(\omega_B) + \Lambda_{\omega_B} \big(\rho_H\big) - \hat S_{\pi} + \Delta_{V} (\Lambda_{\omega_B}(\omega_{M})) \Big) \omega_{M}^m \wedge \omega_B^n  \\
&& + \frac{n}{m+1}\int_M h_B \left(S_{V}(\omega_{M}) - \hat{S}_{V}\right) \omega_{M}^{m+1} \wedge \omega_B^{n-1}  + O\left(k^{-1}\right),
\end{eqnarray*}
where, as in the statement of the theorem, we have suppressed the dependence on $J$ in the scalar curvature quantities and the vertical Laplacian in the final equation. As $h_B$ is pulled back from $B$ and therefore constant along fibres, $\int_M h_B \Delta_{V} (\Lambda_{\omega_B}(\omega_{M})) \omega_{M}^m \wedge \omega_B^n =0$, which implies the result.
\end{proof}

\subsection{A formal space with positivity}\label{sec:formal} The moment map given by Theorem \ref{thm:cplxstrmomentmap} for the $\mathrm{Ham}(\pi)$-action on $\left(\mathscr J_\pi,\Omega_{\mathscr J_\pi}\right)$ is a moment map with respect to a symplectic form which is degenerate, as it vanishes on the mixed terms corresponding to the change of almost complex structure on $X$.  If one changes perspective and varies the K\"ahler form instead of the complex structure, this degeneracy can be seen by the fact that the equation does not change when adding a form pulled back from the base of the submersion $X \to B$ to $\omega$. 

We next construct a quotient of $\left(\mathscr J^{\Int}_\pi,\Omega_{\mathscr J_\pi}\right)$ on which the symplectic form becomes positive. Again we argue formally, and consider spaces which are given by a set together with a specified vector space at each point acting as the formal tangent space.

Define an equivalence relation $\sim$ on $\mathscr J_\pi$ by setting 
$$J\sim J' \iff J_V = J'_V \ \text{and} \ \  J_B = J'_B.$$
Here $J_V$ is the purely vertical part of the almost complex structure $J_M$ on $M$. Thus, $J$ and $J'$ are equivalent precisely when the complex structures on $M$ agree on each fibre of the map $\pi$, and $J_B=J'_B$, which implies that  for $J\sim J'$, the almost complex structures $J_M$ and $J'_M$ only differ in their mixed components. The underling set of our space is then $$\widetilde{\mathscr J}_\pi = \mathscr J_\pi / \mathord{\sim}$$ and we denote by  $\phi : \mathscr J_\pi \to \widetilde{\mathscr J}_\pi$ the quotient map. Now define the tangent space at a point $[J]$ to be the subspace of $T_J \mathscr J_\pi \subset \End(TM)$ generated by purely vertical and purely horizontal endomorphisms. Note that this is well-defined, as $T_J \mathscr J_\pi$ does not depend on $J$.  Thus each tangent space of $\widetilde{\mathscr J}_\pi$ can be identified with a fixed vector space (that we call $W$) that is independent of $[J]$, so the ``tangent bundle" to $\widetilde{\mathscr J}_\pi$ is just the trivial bundle modeled on $W$.

Having defined the tangent spaces at each point of $\widetilde{\mathscr J}_\pi$, we can also define their usual tensor products. We will abuse notation and write $W$ for the tangent bundle of $\widetilde{\mathscr J}_\pi$ as well, and similarly for the tensor bundles.

One may define the exterior derivative $d : \Gamma(\Lambda^j W) \to \Gamma(\Lambda^{j+1} W)$ in the following manner. We say that a $j$-form $\alpha$ on $\widetilde{\mathscr J}_\pi$ is smooth if its pullback $\hat \alpha$ to $\mathscr J_\pi$ is smooth. For such a $j$-form, we then have the exterior derivative $d \hat \alpha$. Over a point $[J] \in \widetilde{\mathscr J}_\pi$, the form $d \hat \alpha$ may depend on the representative. However, since $\hat \alpha$ itself has no mixed terms (as it is pulled back from $\widetilde{\mathscr J}_\pi$), this dependence on $J$ is only in the \emph{mixed} terms of $d \hat \alpha$. Thus, the $\Gamma(\Lambda^{j+1} W)$-component of $d \hat \alpha$ is independent of the representative of $J \in \phi^{-1}([J])$ and  this allows us to define $d \alpha$ at a point $[J]$ as the $\Gamma(\Lambda^{j+1} W)$-component of $d \hat \alpha$ at a representative $J$ in $\mathscr J_{\pi}$.

We obtain a two form $\Omega_{\widetilde{\mathscr J}_\pi} \in \Gamma(\Lambda^2T\widetilde{\mathscr J}_\pi)$ on $\widetilde{\mathscr J}_\pi$ by restricting $\Omega_{\mathscr J_\pi}$ to $W$.  Note that this is well-defined as if $J \sim J'$, then the purely vertical and purely horizontal components of $J$ and $J'$ agree. Hence, the trace terms with respect to $\omega_B$ and $\omega_M$ in the expression for $\Omega_{\mathscr J_\pi}$ will agree for $J$ and $J'$ on elements of $W$. Note also that $\Omega_{\widetilde{\mathscr J}_\pi}$ is closed, since $\phi^* \Omega_{\widetilde{\mathscr J}_\pi} = \Omega_{\mathscr J_\pi}$ and $\Omega_{\mathscr J_\pi}$ is closed. Moreover, $\Omega_{\widetilde{\mathscr J}_\pi}$ is now a positive form, in contrast with $\Omega_{\mathscr J_\pi}$.

Now, $\mathrm{Ham}(\pi)$ acts on $\widetilde{\mathscr J}_\pi$ making the quotient map $\phi :\mathscr J_\pi \to \widetilde{\mathscr J}_\pi$ equivariant. Indeed, if $(\phi_M, \phi_B) \in \mathrm{Ham}(\pi)$ and $J \sim J'$, then $J_B = J_B'$, so $\phi_B\circ J_B = \phi_B \circ J_B'$, and moreover the vertical components of $\phi_M \circ J_M$ and $\phi_M \circ J_M'$ will agree. Thus $(\phi_M \circ J_M, \phi_B \circ J_B)$ and $(\phi_M \circ J'_M, \phi_B \circ J'_B)$ map to the same point in $\widetilde{\mathscr J}_\pi$. As we have the notions of a tangent space, the various tensor bundles and a differential on $\widetilde{\mathscr J}_\pi$, it makes sense to ask for the $\mathrm{Ham}(\pi)$-action on $\widetilde{\mathscr J}_\pi$ to be Hamiltonian. All the terms in the expression in the statement of Theorem \ref{thm:cplxstrmomentmap} for $\mu_\pi(J)$ are independent of $J \in \phi^{-1}([J])$. Thus we can define $\mu_{\pi} : \widetilde{\mathscr J}_\pi \to \mathfrak{ham}(\pi)^*$ by its value on a representative in $\mathscr J_\pi$ of a given point in $\widetilde{\mathscr J}_\pi$. Unraveling the definitions, we see that the moment map property holds on $\widetilde{\mathscr J}_\pi$, since it does so on $\mathscr J_{\pi}$. Thus, one may view the equation of finding a zero of $\mu_{\pi}$ as a moment map equation with respect to a positive form, provided we pass to the space $\widetilde{\mathscr J}_\pi$.

With the definitions given above, we have shown the following:
\begin{corollary}
    The $\mathrm{Ham}(\pi)$-action on $\widetilde{\mathscr J}_\pi$ is Hamiltonian with moment map $\mu_{\pi}$ with respect to $\Omega_{\widetilde{\mathscr J}_\pi}$, which is a positive form on $\widetilde{\mathscr J}_\pi$.
\end{corollary}

\subsection{Zeroes of the moment map}

We next wish to characterise zeroes of the moment map $\mu_{\pi}$. The group $\mathrm{Ham}(\pi)$ fits into an exact sequence
\begin{align}
    \label{eq:ES}
    0 \to \mathrm{Ham}_0 (M, \omega) \to \mathrm{Ham}(\pi) \to \mathrm{Ham}(B, \omega_B),
\end{align}
where the Lie algebra $\fham_0(M,\omega)$ of $\mathrm{Ham}_0 (M, \omega)$ consists of the fibrewise Hamiltonian vector fields that are purely vertical, while the Lie algebra $\fham(B,\omega_{B})$ of $\mathrm{Ham}(B, \omega_B)$ consists of the Hamiltonian vector fields on $B$. Note that a Hamiltonian vector field on $B$ may not lift to $X$ as a Hamiltonian vector field. This would require the vector field to lift and to admit a Hamiltonian with respect to the form $\omega$.  So, the sequence may not be a short exact sequence.

Now, we can identify the Lie algebra $\fham_0(M,\omega)$ with the space $C^{\infty}_0(M)$ of functions of average $0$ on each fibre of $M \to B$. We do this by associating to such a vector field $v$ the function obtained by gluing together the fibrewise Hamiltonians $h_b$ of average $0$ for $v_b$ on $M_b$. Since $\mathrm{Ham}_0 (M, \omega)$ is a subgroup, this means that a zero of the moment map in particular satisfies 
\begin{eqnarray*}
\int_M h_{M} \left(S_{V}(\omega)-\hat S_{V} \right) \omega^m \wedge \beta^n  = 0
\end{eqnarray*}
for every $h_{M} \in C^{\infty}_0(M)$. This implies that $S_{V}(\omega)-\hat S_{V} =0$, i.e. at a zero of the moment map, $\omega$ is fibrewise cscK. This further implies that at a zero of the moment map, the moment map has expression
$$
\int_M h_B (\Lambda_{\beta}( \Ric(\beta) + \rho_H) - \hat S_{\pi}) \omega^m \wedge \beta^n.$$
Thus, at a zero of the moment map the above is zero for all $(h_{M},h_B) \in \fham(\pi)$, which means that the twisted scalar curvature $\Lambda_{\beta}( \Ric(\beta) + \rho_H)$ lies in the orthogonal complement  $\mathfrak{k}$ to the image of the map
$$
\fham(\pi) \to \fham(B,\omega_B).
$$
Thus a zero of the moment map is the pair of a fibrewise cscK metric on $M \to B$ together with a base metric whose twisted scalar curvature lies in $\mathfrak{k}$.
\begin{corollary}
    $J$ is a zero of the moment map $\mu_{\pi}$ if and only if $(J_M, \omega_M)$ is a fibrewise cscK metric and the twisted scalar curvature of $(J_B, \omega_B)$ lies in $\mathfrak{k}$.
\end{corollary}

\begin{remark}
    If $X \to B$ admits a fibrewise cscK metric, we can view the equation as an equation for $\omega_B$. The fibrewise cscK metric may not be unique, but the equation for the base metric is in fact  independent of the choice of fibrewise cscK metric, as the Weil--Petersson forms associated to two relatively cscK metrics agree, see \cite[Corollary 4.3]{DervanNaumann}.
\end{remark}

\section{Foliations}\label{sec:foliations}

We now generalise the above to the case of foliations. 
Let $\cF \subset TX$ be a foliation, $\omega$ a leafwise symplectic form and $\beta$ a transversal symplectic form. We will define a space $\mathscr J^{\Int}_{\cF}$ of integrable almost complex structures compatible with the foliation $\cF$ and endow it with a K\"ahler structure so that the coupled quantity involving the leafwise scalar curvature of $\omega$ and the twisted scalar curvature of $\beta$ becomes a moment map with respect to a certain group action. 

The approach mirrors that of the submersion case closely and is in particular based on an adiabatic argument. There are several reasons why the case of foliations is treated separately from the case of submersions. First, the spaces, groups and so on that are considered in this section have more concrete interpretations in the submersion case.  Second, the twisted scalar curvature equation becomes an equation on $X$ in the foliation case, whereas it is an equation on the base $B$ when we have a foliation, see Corollary \ref{cor:zeroofmomentmap}. Finally, in the submersion case, we have an interpretation of having a zero of the moment map in terms of the twisted cscK equation involving the Weil--Petersson form as twist. This comes from the fact when there is a fibrewise cscK metric, the submersion induces a map from the base $B$ to the moduli space of cscK manifolds, and the twisted equation on the base then has the Weil--Petersson form pulled back via this map as its twist. 

We begin with the definitions of the relevant space of almost complex structures. Fix a compact symplectic manifold $(M,\omega)$, a foliation $\cF\subset TM$ of the (smooth) tangent bundle of $M$, and a transverse closed two-form $\beta$.
\begin{definition}\label{pi-var2}
We define the space of \emph{integrable almost complex structures} on $\cF$ to be 
\begin{align*}
\mathscr J^{\Int}_{\cF} =\left\{J\in\mathscr J^{\Int}(M,\omega) :
\cF \text{ is $J$-holomorphic and }  J^*\beta=\beta\right\}.
\end{align*}
\end{definition}
 
In particular, $J$ sends $\cF$ to itself, which means its $\textnormal{Hom}(\cF, H)$-component vanishes. Moreover, as a consequence, one has
$$
\mathscr J^{\Int}_{\cF} \subset \mathscr J(M,\omega+k\beta),
$$
just as in Lemma \ref{lem:embeddings}. Note that $\mathscr J^{\Int}_{\cF}$ depends on both $\omega$ and $\beta$, even though this is not reflected in the notation. 

A foliation is locally a holomorphic submersion. We employ a metric enhancement of this classical fact, for almost complex structures $J \in \mathscr J^{\Int}_{\cF}$, so that $J$ preserves $\beta\in \Omega^{2}_H(M)$, which is hence a $(1,1)$-form on the complex manifold $(M,J)$.   Choose a holomorphic neighbourhood $U$ around any point of $X$ with coordinates $x= (y,z) \in \mathbb{C}^{m+n}$, such that leaves of $\cF$ are precisely the fibres of the local submersion $x \mapsto z$. Then, since $d\beta = 0$, the form $\beta$ is locally the pullback of a K\"ahler metric from the base of the local submersion. Thus, for each $J \in J^{\Int}_{\cF}$, the induced metrised foliation is locally a metrised holomorphic submersion.  This being a local claim, we illustrate it with the following example, with the general case being identical.

\begin{example}
Let $\beta\in \Omega^{1,1}(\mathbb{C}^2)$ be closed, and horizontal with respect to the submersion $\mathbb{C}^2\to\mathbb{C}$ given by $(z_1,z_2) \mapsto z_2$. We may write 
$$
\beta=f_{1\bar 1}dz_1\wedge d\bar z_1+f_{1\bar 2} dz_1\wedge d\bar z_2 + f_{2\bar 1}dz_2\wedge d\bar z_1+f_{2\bar 2}dz_2\wedge d\bar z_2
$$
The horizontal condition is $ \iota_{\partial_{z_1}} \beta=\iota_{\partial_{\bar z_1}} \beta=0$, which implies that  $\beta=f_{2\bar 2}dz_2\wedge d\bar z_2$. Moreover, $d\beta=0$ implies that $f_{2\bar 2}$ depends only on the $z_2$-coordinate, as the derivatives in the $z_1$ and $\bar z_1$-directions then vanish. Thus, $\beta$ is the pullback of the induced form on $\mathbb{C}$ via the submersion $\mathbb{C}^2 \to \mathbb{C}$.

\end{example}

Next, we define the relevant gauge group that acts on $\mathscr J^{\Int}_{\cF}$. We first define a subgroup of $\mathrm{Ham}(X,\omega)$, which does not depend on $\beta$, as follows.  First set
\begin{align*}
\mathrm{Ham}(X,\omega,\cF):=\left\{ \phi\in\mathrm{Ham}(X,\omega) \left |  \phi^\ast \cF= \cF\right.\right\},
\end{align*}
where $\phi^\ast \cF\subset TX$ is regarded as the pullback subbundle of $\cF\subset TX$, meaning we have a diagram
$$
\xymatrix{
\phi^\ast\cF\ar@{>}[rr]^{} \ar@{>}[d]^{}& & \cF  \ar@{>}[d]^{} \\
 	X\ar@{>}[rr]^{\phi}   & & X.  }
$$
The group of Hamiltonian symplectomorphisms of $\cF$ will be a subgroup of this group.
\begin{definition} We define the group of \emph{Hamiltonian symplectomorphisms} $\Ham(\cF)$ of $(X,\cF,\omega,\beta)$ to be the elements of $\mathrm{Ham}(X,\omega,\cF)$ that also preserve $\beta$, i.e.
$$
\mathrm{Ham}(\cF)=\left\{ \phi\in\mathrm{Ham}(X,\omega,\cF) \left |  \phi^\ast\beta=\beta\right.\right\}.
$$
\end{definition}
Note that this depends on both $\omega$ and $\beta$, but as with $\mathrm{Ham}(\pi)$ in the case of submersions we omit this from the notation. The Lie algebra of $\mathrm{Ham}(\cF)$ is given by
$$
\mathfrak{h}_{\cF}=\left\{ v \in\mathfrak{ham}(X,\omega,\cF) \left |  \iota_{v}\beta=df \text{ for some }f\in C^\infty_H(X) \right.\right\}.
$$
\begin{remark}
    Note that in the case when the foliation arises from a holomorphic submersion $\pi: X\to B$ and $\beta=\pi^\ast \omega_B$, the Hamiltonian $f=\pi^*\tilde f$ for $v$ with respect to $\beta$ is pulled back from $B$. The Lie algebra $\mathfrak{h}_{\cF}$ is then given by
$$
\left\{ v \in \fham(X,\omega,\cF) : \pi^* \iota_{\pi_* v} \omega_B = \iota_{v}\beta = \pi^* d \tilde f \text{ for some } \tilde f\in C^\infty(B) \right\}.
$$
So in the case of a submersion $\mathfrak{h}_{\cF}$ agrees with $\mathfrak{h}_{\pi}$.
\end{remark}

We can view $\mathscr J^{\Int}_{\cF}$ as a symplectic manifold by expanding the restriction of the symplectic form of $\mathscr J^{\Int} (X,\omega_k)$ to $\mathscr J^{\Int}_{\cF}$ asymptotically in $k$. The leading order term in this expansion is of order $k^n$ and is given by

$$
\Omega_{\mathscr J_\cF}(A^1, A^2) 
=\int_X\left(\beta(A^1, A^2)+ \omega_{\cF}(A^1, A^2)\right)\omega^{m}\wedge \beta^n,
$$
for $A^j \in T_{(J_M, J_B)}\mathscr J^{\Int}_{\cF}$. As $\Omega_{\mathscr J_\cF}$ arises from an asymptotic expansion of a sequence K\"ahler forms, this form is then  closed and semi-positive on $\mathscr J^{\Int}_{\cF}$ just as was the case in Lemma \ref{lem:Omegaclosed}, but again it is not in general positive.

As a consequence of the calculation following Definition \ref{def:transversefutaki}, we obtain the following, with the analogous notation:
\begin{theorem}
Let $v \in \mathfrak h_{\cF}$, and let $h_{\cF}$ be the Hamiltonian for $v$ with respect to $\omega$ and $h$ be the Hamiltonian with respect to $\beta$, as in Equation \eqref{h_k}. A moment map for the action of $\Ham(\mathcal F)$ on $\left(\mathscr{J}^{\Int}_{\cF},\Omega_{\mathscr J_{\cF}}\right)$ is given by the map $\mu_{\cF}: \mathscr{J}^{\Int}_{\cF} \to \mathfrak{h}_{\cF}^\ast$ defined by
\begin{align*}
\langle\mu_{\cF}(J),v\rangle =& \int_X h_{\cF} \left(S_{\cF}(\omega)-\hat S_{\cF} \right) \omega^m \wedge \beta^n + \int_X h (\Lambda_{\beta}( \Ric(\beta) + \rho) - \lambda) \omega^m \wedge \beta^n \\
&+ \frac{n}{m+1}\int_X h (S_{\cF}(\omega) -\hat S_{\cF}) \omega^{m+1} \wedge \beta^{n-1}.
\end{align*}
\end{theorem}
This follows in exactly the same way as Theorem \ref{thm:cplxstrmomentmap}, using the sequence of K\"ahler metrics $\omega + k \beta$ on $X$ and the corresponding expansion of the associated moment map on the space of almost complex structures on $X$. Note that we may further also define a formal quotient space $\widetilde{\mathscr J}_{\cF}$ on which we obtain a positive symplectic form for which $\mu$ is a moment map, as in Section \ref{sec:formal}. We note in addition that it is straightforward to see that this result generalises the corresponding result for submersions.

Finally, a zero of the moment map $\mu_{\cF}$ has a similar interpretation in terms of the twisted scalar curvature.  We have a subgroup $\mathrm{Ham}_0 (M, \omega_{\cF}) \subset \mathrm{Ham}(\cF)$ generated by the vector fields that are sections of $\cF \subset TX$.  We can then define a space $\mfk$ which is the cokernel of the map from $\mathfrak{h}_{\cF}$ to the horizontal Hamiltonian functions with respect to $\beta$. The vanishing of the Futaki invariant then in particular implies the vanishing of 
$$
\int_X h_{\cF} \left(S_{\cF}(\omega)-\hat S_{\cF} \right) \omega^m \wedge \beta^n
$$
for all $h_{\cF} \in \mathfrak h_{\cF}$. As this equals the integral of the inner product of the leafwise gradients of $h_{\cF}$ and $S_{\cF}(\omega)$ over $M$ (with  volume form $\omega^m \wedge \beta^n$), this forces the leafwise gradient $\nabla_{\cF} S_{\cF}(\omega)$ of the leafwise scalar curvature to vanish if there is a zero of the moment map, and so $S_{\cF}(\omega)$ must be constant as it has no transversal component. This in turn implies that at a zero of the moment map, the twisted transversal scalar curvature must land in $\mfk$.
\begin{corollary}\label{cor:zeroofmomentmap}
    $J$ is a zero of the moment map $\mu_{\cF}$ if and only if $(J, \omega_{\cF})$ is a leafwise cscK metric and the twisted scalar curvature of $(J, \beta)$ lies in $\mathfrak{k}$.
\end{corollary}

Note that the resulting condition on $(J,\beta)$ is an equation on the total space of the foliation, in contrast to the setting of submersions, where the condition is one formulated on the base $B$. By analogy, the condition that the twisted scalar curvature of $(J, \beta)$ lies in $\mathfrak{k}$ can be viewed as a basic condition.

\bibliography{twisted}{}
\bibliographystyle{habbrv}

\end{document}